\documentclass[11pt]{amsart}
\usepackage{mathrsfs,eucal,amsmath,amssymb,latexsym,dsfont,stmaryrd,enumerate}

\usepackage[all,2cell]{xy} \UseAllTwocells \SilentMatrices
\usepackage[T1]{fontenc}
\usepackage{xcolor}

\usepackage[pdftex]{hyperref}

\begin{document}

\newcommand{\INVISIBLE}[1]{}

\newtheorem{thm}{Theorem}[section]
\newtheorem{lem}[thm]{Lemma}
\newtheorem{cor}[thm]{Corollary}
\newtheorem{prp}[thm]{Proposition}
\newtheorem{conj}[thm]{Conjecture}

\theoremstyle{definition}
\newtheorem{dfn}[thm]{Definition}
\newtheorem{question}[thm]{Question}
\newtheorem{nota}[thm]{Notations}
\newtheorem{notation}[thm]{Notation}
\newtheorem*{claim*}{Claim}
\newtheorem{ex}[thm]{Example}
\newtheorem{counterex}[thm]{Counter-example}
\newtheorem{rmk}[thm]{Remark}
\newtheorem{rmks}[thm]{Remarks}
 
\def\labelenumi{(\arabic{enumi})}

\newcommand{\aro}{\longrightarrow}
\newcommand{\arou}[1]{\stackrel{#1}{\longrightarrow}}
\newcommand{\RA}{\Longrightarrow}

\newcommand{\mm}[1]{\mathrm{#1}}
\newcommand{\bm}[1]{\boldsymbol{#1}}
\newcommand{\bb}[1]{\mathbf{#1}}

\newcommand{\bA}{\boldsymbol A}
\newcommand{\bB}{\boldsymbol B}
\newcommand{\bC}{\boldsymbol C}
\newcommand{\bD}{\boldsymbol D}
\newcommand{\bE}{\boldsymbol E}
\newcommand{\bF}{\boldsymbol F}
\newcommand{\bG}{\boldsymbol G}
\newcommand{\bH}{\boldsymbol H}
\newcommand{\bI}{\boldsymbol I}
\newcommand{\bJ}{\boldsymbol J}
\newcommand{\bK}{\boldsymbol K}
\newcommand{\bL}{\boldsymbol L}
\newcommand{\bM}{\boldsymbol M}
\newcommand{\bN}{\boldsymbol N}
\newcommand{\bO}{\boldsymbol O}
\newcommand{\bP}{\boldsymbol P}
\newcommand{\bY}{\boldsymbol Y}
\newcommand{\bS}{\boldsymbol S}
\newcommand{\bX}{\boldsymbol X}
\newcommand{\bZ}{\boldsymbol Z}

\newcommand{\cc}[1]{\mathcal{#1}}
\newcommand{\ccc}[1]{\mathscr{#1}}

\newcommand{\ca}{\cc{A}}

\newcommand{\cb}{\cc{B}}

\newcommand{\cC}{\cc{C}}

\newcommand{\cd}{\cc{D}}

\newcommand{\ce}{\cc{E}}

\newcommand{\cf}{\cc{F}}

\newcommand{\cg}{\cc{G}}

\newcommand{\ch}{\cc{H}}

\newcommand{\ci}{\cc{I}}

\newcommand{\cj}{\cc{J}}

\newcommand{\ck}{\cc{K}}

\newcommand{\cl}{\cc{L}}

\newcommand{\cm}{\cc{M}}

\newcommand{\cn}{\cc{N}}

\newcommand{\co}{\cc{O}}

\newcommand{\cp}{\cc{P}}

\newcommand{\cq}{\cc{Q}}

\newcommand{\cR}{\cc{R}}

\newcommand{\cs}{\cc{S}}

\newcommand{\ct}{\cc{T}}

\newcommand{\cu}{\cc{U}}

\newcommand{\cv}{\cc{V}}

\newcommand{\cy}{\cc{Y}}

\newcommand{\cw}{\cc{W}}

\newcommand{\cz}{\cc{Z}}

\newcommand{\cx}{\cc{X}}

\newcommand{\g}[1]{\mathfrak{#1}}

\newcommand{\af}{\mathds{A}}
\newcommand{\PP}{\mathds{P}}

\newcommand{\GL}{\mathrm{GL}}
\newcommand{\PGL}{\mathrm{PGL}}
\newcommand{\SL}{\mathrm{SL}}
\newcommand{\NN}{\mathds{N}}
\newcommand{\ZZ}{\mathds{Z}}
\newcommand{\CC}{\mathds{C}}
\newcommand{\QQ}{\mathds{Q}}
\newcommand{\RR}{\mathds{R}}
\newcommand{\FF}{\mathds{F}}
\newcommand{\DD}{\mathds{D}}
\newcommand{\VV}{\mathds{V}}
\newcommand{\HH}{\mathds{H}}
\newcommand{\MM}{\mathds{M}}
\newcommand{\OO}{\mathds{O}}
\newcommand{\LL}{\mathds L}
\newcommand{\BB}{\mathds B}
\newcommand{\kk}{\mathds k}
\newcommand{\bs}{\mathbf S}
\newcommand{\GG}{\mathds G}
\newcommand{\XX}{\mathds X}

\newcommand{\WW}{\mathds W}
\newcommand{\al}{\alpha}

\newcommand{\be}{\beta}

\newcommand{\ga}{\gamma}
\newcommand{\Ga}{\Gamma}

\newcommand{\om}{\omega}
\newcommand{\Om}{\Omega}

\newcommand{\vt}{\vartheta}
\newcommand{\te}{\theta}
\newcommand{\Te}{\Theta}

\newcommand{\ph}{\varphi}
\newcommand{\Ph}{\Phi}

\newcommand{\ps}{\psi}
\newcommand{\Ps}{\Psi}

\newcommand{\ep}{\varepsilon}

\newcommand{\vr}{\varrho}

\newcommand{\de}{\delta}
\newcommand{\De}{\Delta}

\newcommand{\la}{\lambda}
\newcommand{\La}{\Lambda}

\newcommand{\ka}{\kappa}

\newcommand{\si}{\sigma}
\newcommand{\Si}{\Sigma}

\newcommand{\ze}{\zeta}

\newcommand{\fr}[2]{\frac{#1}{#2}}
\newcommand{\vs}{\vspace{0.3cm}}
\newcommand{\na}{\nabla}
\newcommand{\pd}{\partial}
\newcommand{\po}{\cdot}
\newcommand{\met}[2]{\left\langle #1, #2 \right\rangle}
\newcommand{\rep}[2]{\mathrm{Rep}_{#1}(#2)}
\newcommand{\repo}[2]{\mathrm{Rep}^\circ_{#1}(#2)}
\newcommand{\hh}[3]{\mathrm{Hom}_{#1}(#2,#3)}
\newcommand{\modules}[1]{#1\text{-}\mathbf{mod}}
\newcommand{\vect}[1]{#1\text{-}\mathbf{vect}}

\newcommand{\Modules}[1]{#1\text{-}\mathbf{Mod}}
\newcommand{\dmod}[1]{\mathcal{D}(#1)\text{-}{\bf mod}}
\newcommand{\spc}{\mathrm{Spec}\,}
\newcommand{\an}{\mathrm{an}}
\newcommand{\NNo}{\NN\smallsetminus\{0\}}

\newcommand{\pos}[2]{#1\llbracket#2\rrbracket}

\newcommand{\lau}[2]{#1(\!(#2)\!)}

\newcommand{\cpos}[2]{#1\langle#2\rangle}

\newcommand{\id}{\mathrm{id}}

\newcommand{\one}{\mathds 1}

\newcommand{\ti}{\times}
\newcommand{\tiu}[1]{\underset{#1}{\times}}

\newcommand{\ot}{\otimes}
\newcommand{\otu}[1]{\underset{#1}{\otimes}}

\newcommand{\wh}{\widehat}
\newcommand{\wt}{\widetilde}
\newcommand{\ov}[1]{\overline{#1}}
\newcommand{\un}[1]{\underline{#1}}

\newcommand{\op}{\oplus}

\newcommand{\lid}{\varinjlim}
\newcommand{\lip}{\varprojlim}

\newcommand{\mcrs}{\bb{MC}_{\rm rs}}
\newcommand{\mclog}{\bb{MC}_{\rm log}}
\newcommand{\mc}{\mathbf{MC}}

\newcommand{\ega}[3]{[EGA $\mathrm{#1}_{\mathrm{#2}}$, #3]}

\newcommand{\asts}{\begin{center}$***$\end{center}}
\title[Connections with parameters]{Algebraic theory of formal regular-singular connections with parameters}

\author[P. H. Hai]{Ph\`ung H\^o Hai}

\address{Institute of Mathematics, Vietnam Academy of Science and Technology, Hanoi, 
Vietnam}

\email{phung@math.ac.vn}

\author[J. P. dos Santos]{Jo\~ao Pedro dos Santos}

\address{Institut de Math\'ematiques de Jussieu -- Paris Rive Gauche, 4 place Jussieu,
Case 247, 75252 Paris Cedex 5, France}

\email{joao\_pedro.dos\_santos@yahoo.com}

\author[P. T.  Tam]{Ph\d am Thanh T\^am}

\address{Department of Mathematics, Hanoi Pedagogical University 2, Vinh Phuc, Vietnam}
\email{phamthanhtam@hpu2.edu.vn}


\begin{abstract}This paper is divided into two parts. The first is a review, through categorical lenses, of the classical theory of regular-singular differential systems over $\lau Cx$ and $\PP^1_C\smallsetminus\{0,\infty\}$, where $C$ is algebraically closed and of characteristic zero. It aims at reading the existing classification results as an equivalence between regular-singular systems and representations of the group $\ZZ$. In the second part, we deal with regular-singular connections over $\lau Rx$ and $\PP_R^1\smallsetminus\{0,\infty\}$, where $R=\pos C{t_1,\ldots,t_r}/I$. The picture we offer shows that regular-singular connections are equivalent to representations of $\ZZ$, now over $R$.  
\end{abstract}

\date{\today}
\maketitle

\section{Introduction}

This paper is an outgrowth of our study of regular-singular connections through the past years.  It is divided into two parts which although thematically close, are distinct in originality.  Indeed, Part I is a patient revision of classical theory  (\cite[Chapter 4]{coddington-levinson55}, \cite{manin65}, \cite{wasow76}, \cite[Section 16]{ilyashenko-yakovenko08}) of regular-singular connections (or differential systems) in a more categorical setting {\it plus} an exposition of a more recent original contribution of Deligne   \cite[\S15]{deligne87}. 
Part II is a study of the theory of regular-singular connections  on $\lau Rx$ and on $\PP^1_R\smallsetminus\{0,\infty\}$, where $R$ is a certain complete local ring. The method behind Part II comes in great part from \cite{hai-dos_santos20} and it is hoped that it will be a means to grasp op.cit. in a less complex-analytic setting.

The classical theory of formal regular-singular connections  presents roughly two classifications of these objects: one by reducing each system to one with constant coefficients (\cite{coddington-levinson55}, \cite{wasow76}, \cite{ilyashenko-yakovenko08}), and one by means of tensor products of  unipotent  and  rank-one connections \cite{manin65}.  As beautiful as they are, these classifications tend to give an incomplete picture due to the lack of categorical structures and equivalences. For example,  although systems of differential equations with constant coefficients play a fundamental role, their natural properties are seldom addressed.  
Our take on the matter, accomplished in Part I, is to use \cite{deligne70} {\it as guiding principle} and obtain an equivalence between formal regular-singular systems and representations of the ``fundamental group'', which is $\ZZ$. 
As far as we know, this point of view is adopted, over $\CC$, only in \cite{van_der_put-singer03}.
In addition, under this mindset, we are able to comment on the important theory of Deligne's tensor product of categories. Our approach to the theory of connections on $\PP^1\smallsetminus\{0,\infty\}$ follows the same path, but its structuring is facilitated by the formal case.

Part II contains new material on formal differential modules whose ring of constants is a complete local ring. Our original motivation for writing down this piece was to give a less technical  and algebraic version of our paper \cite{hai-dos_santos20} which, nevertheless, would allow us to see the main ideas in loc.cit. To wit, an abstract  picture stemming from \cite{hai-dos_santos20}
is this. Let $C$ be an algebraically closed field of characteristic zero,   $R$ a noetherian, local and complete $C$-algebra with maximal ideal $\g r$ and residue field  $C$. 
We now give ourselves two $R$-linear categories $\cC$ and $\cC'$; denote by   $\cC_{n}$ and $\cC_{n}'$ be the full subcategories of objects ``annihilated by $\g r^{n+1}$''. Now, suppose that $\cC_{0}\simeq \cC_{0}'$. We wish to conclude that $\cC\simeq \cC'$.  
The strategy is to promote $\cC_{0}\simeq \cC_{0}'$ into an equivalence $\cC_{n}\simeq \cC_{n}'$ for all $n$ and then to ``pass to the limit''. (Needless to say, this is only reasonable in certain cases.) Part II of the present work goes through this idea in the special case where $\cC$ is the category of regular-singular formal connections and $\cC'$ is the category of representations of the abstract group $\ZZ$. The equivalence between $\cC_0$ and $\cC_0'$ is derived here from the results of Part I, while   in \cite{hai-dos_santos20} we relied on \cite{deligne70}. 

Let us now review the remaining sections separately. In what follows, $C$ is an algebraically closed field of characteristic zero and for any $C$-algebra  $R$, we let $\vt$ stand for the derivation of $\lau Rx=\pos Rx[x^{-1}]$ defined by $\vt\sum a_kx^k=\sum ka_kx^k$.

Section \ref{08.05.2020--1} serves to introduce  basic notations and definitions: specially 
important are the {\it logarithmic connections}  and the {\it regular-singular} ones over $\lau Cx$, see Definitions \ref{08.04.2020--1} and \ref{15.04.2020--1}. Section \ref{11.09.2020--1} covers basic facts on {\it Euler connections}, which correspond to  differential systems of the form $\vt y =Ay$ in which $A$ is a matrix with entries on $C$ (Definition \ref{12.06.2020--1}). The approach is {\it categorical} and we study   the {\it Euler functor} from the category of ``endomorphisms'' to the category of logarithmic connections (Definition \ref{12.06.2020--5}). Most findings  contain little more than simple remarks on  spectral analysis of linear operators in finite dimension.

Section \ref{08.04.2020--3} brings to light one of the main actors in the whole theory: the {\it residue endomorphism} of a logarithmic connection. Most results of this part are well-known, although not phrased in our language (see Theorems \ref{03.03.2020--1} and \ref{02.03.2020--3}). 
But not all is referencing, and in Proposition \ref{09.04.2019--1} we show, motivated by our categorical take, how to limit ``the size of poles'' between an arrow of logarithmic models in terms of the difference of the exponents. This plays later an important role when dealing with regular-singular connections ``depending on parameters'' (e.g.   the proof of Theorem \ref{16.05.2020--1}).  The section ends with the construction of preferred logarithmic models of  regular-singular connections (Theorem \ref{06.07.2020--4}); we name these Deligne-Manin models, but many other names are    in the literature (canonical extensions, $\tau$-extensions, etc).  

Section \ref{25.06.2021--1} revisits Manin's elegant paper \cite{manin65} with the intention of presenting its gist as an equivalence between the categories of representations of $\ZZ$ and regular-singular connections. It begins by using  classical results to prove a fundamental structural theorem of \cite{manin65} and then goes on to study {\it unipotent} (Section \ref{17.09.2020--5})  and {\it diagonalizable} (Section \ref{17.09.2020--4}) regular-singular connections. The former category is    then proved to be equivalent to the category of unipotent endomorphism (see Theorem \ref{25.06.2021--4}); this allows us to observe that unipotent regular-singular  connections amount to representations of the additive group (Corollary \ref{15.09.2020--2}). We go on to exhibit an equivalence between the category of diagonalizable  connections and representations of the diagonal group scheme whose group of characters is $C/\ZZ$. Calling on set theory, we note that $C/\ZZ\simeq C^\ti$, which puts us in an ideal position to establish an equivalence between regular-singular connections and representations of $\ZZ$. This final goal  is obtained by means of the  Deligne tensor product of abelian categories. This construction is a delicate piece of category theory so that   some of the necessary results are to be written down in a separate work \cite{dos_santos21tp}. Here, we content ourselves with a brief presentation of the definitions and fundamental results (Section \ref{17.09.2020--1}). In Section \ref{24.06.2021--1} all is put together to arrive at the conclusion  motivating  the section, which is Corollary \ref{22.04.2020--2}.

 With Section \ref{01.07.2020--5} we end Part I with a  review of  an equivalence  between regular-singular connections on $\lau  Cx$
and on $\PP^1\smallsetminus\{0,\infty\}$ (Theorem \ref{12.06.2020--2}). We follow mostly the ideas in \cite[15.28--36]{deligne87} in proving  the key non-trivial point: all regular-singular connections on $\PP^1\smallsetminus\{0,\infty\}$ are ``Euler connections'', see Proposition \ref{09.06.2020--1}. 
From that and the knowledge obtained in the previous sections the desired equivalence follows without much effort.

We now begin to review the sections pertaining to Part II. In Section \ref{28.06.2021--1}, we fix a certain finite dimensional $C$-algebra $\La$ and start exploring the notion of objects in $C$-linear categories carrying an action of $\La$ (Definition \ref{28.06.2021--2}). 
This is to be applied to  categories of regular-singular connections and we show that most results from Part I carry over to this context. See for example the existence of Deligne-Manin models stated in Theorems \ref{10.03.2020--2} and   \ref{10.07.2020--3}.  Let us draw the reader's attention to the notion of freeness in relation  to $\La$,  see Definitions \ref{16.10.2020--1} and \ref{16.10.2020--2}, which plays a key role in the rest of the paper. 

In Section  \ref{28.06.2021--5}, after fixing a {\it complete local noetherian $C$-algebra $R$} having residue field $C$, we begin the study of {\it regular-singular connections over $\lau Rx$}. One of the most relevant concepts in this case is our definition of {\it residues and exponents} (Definition \ref{28.06.2021--6}) stating that ``exponents should be indifferent to reduction modulo   the maximal ideal of $R$''.  In particular,     exponents are   elements of       $C$. This definition allows us to prove Theorem \ref{16.05.2020--5}, the analogue of Theorem \ref{03.03.2020--1}, which shows that Euler connections still play a central role in this theory. Then, applying ideas around the theme of Hensel's Lemma, we explain how to lift the Jordan decomposition of an endomorphism between $R$-modules (Corollary \ref{17.05.2020--1}), which in turn allows us to deduce Theorem \ref{25.05.2020--1}, paralleling  Theorem \ref{02.03.2020--3}  in the present context. 
At this point, our  assumptions  on the $\pos Rx$-modules are in many places  strong---they are to be free---and improvements appear in Section \ref{28.06.2021--7}. 
We also draw attention to Theorem \ref{29.06.2020--1} and Remark \ref{09.09.2020--1}. In the former result, we present a criterion for a connection over $\lau Rx$ to underlie a flat $\lau Rx$--module. Since the fibres of $\spc {\lau Rx}\to\spc R$ are not generally of finite type over the residue field, the proof of Theorem \ref{29.06.2020--1}  relies on a beautiful result of Y.  Andr\'e,  which we reprove swiftly in Remark \ref{09.09.2020--1}.

Section \ref{28.06.2021--7} contains the first  main result,  Corollary \ref{16.07.2020--1}.  It shows the equivalence 
\[\tag{$*$}
\begin{array}{c}\text{   regular-singular connections} \\ \text{  over $\lau Rx$}
\end{array}
\xymatrix{\ar[rr]^-{\sim}&&}
\text{$R$-representations of $\ZZ$},  
\]  
  thus obtaining  the exact analogue of Deligne-Manin's theory from Section \ref{25.06.2021--1}. (No   assumption is made on the nature of the  $\lau Rx$ or $R$-modules underlying connections or representations.)
The heart of the matter is the existence of certain preferred logarithmic models (Deligne-Manin) for regular-singular connections over $\lau Rx$ and these are obtained in Theorem \ref{16.05.2020--1}. The proof of this result relies on the fact that we are able to ``pass to the limit'' of the models obtained previously---since $\pos Rx$ is a complete local ring---to construct a suitable logarithmic model. Such a limit process is only possible 
since exponents do not change from ``truncation to truncation'' and since the ``size of the pole'' of a given arrow is controlled by the differences of exponents (Proposition \ref{09.04.2019--1}). To see what can easily go wrong, the reader should read Counter-example \ref{28.06.2021--8}. 
Once the logarithmic models of Theorem \ref{16.05.2020--1} have been shown to exist, we are then able to apply a limit process to arrive at   the equivalence
$(*)$. 

The paper then ends with Section \ref{29.06.2021--1},  which shows a second main result: The restriction functor 
\[\tag{$**$}
\begin{array}{c}\text{ regular-singular connections} \\ \text{  over $\PP_R\smallsetminus\{0,\infty\}$}
\end{array}
\xymatrix{\ar[rr]^-{\text{restriction}}&&}
\begin{array}{c}\text{category of regular-singular} \\ \text{connections over $\lau Rx$}
\end{array}
\]  
is an equivalence.  (See Theorem \ref{20.10.2020--6}.) The proof is based on the previous techniques with one important modification: the fact that modules over $\pos Rx$ are constructed from limits leaves place to Grothendieck's GFGA, stating that coherent modules over $\PP_R$  are constructed by limits of  coherent modules over the truncations of $\PP_R$ modulo the maximal ideal. 

Finally, let us call the reader's attention to some important  works on ``differential structures depending on parameters'' which have appeared in recent times: these are \cite{nitsure1}, \cite{mfs0}, \cite{mfs},  \cite{fmfs}  and \cite{fmfs2}. At the end of   the introduction in \cite{hai-dos_santos20}, the reader shall find a brief summary of some of the ideas behind these works. 

\subsection*{Acknowledgements}
The research of the first and   third named authors is funded by the Vietnam Academy of Science and Technology under grant number NVCC01.01/22-23. The second named author wishes to thank I. Lopez-Franco for instructing him on the theory of the Deligne tensor product. He also profits to thank the CNRS for granting him a research leave from September to December 2020, when some parts of this text were elaborated. 
  
\subsection*{Notation and conventions}
\begin{enumerate}[(1)]\item In this text,  $C$ stands for an   algebraically closed field of characteristic zero.   
\item Given a (commutative and unital) ring $R$, we let $\lau Rx$ stand for $\pos Rx[x^{-1}]$  and $\vt:\lau Rx\to\lau Rx$ the derivation defined by  
$$\vt\sum a_nx^n=x\frac{d}{dx} \sum a_nx^n=\sum na_nx^n.$$
\item We let $\mm M_{m\ti n}(R)$, respectively $\mm M_n(R)$, stand for the associative ring of $m\ti n$ matrices, respectively $n\ti n$ matrices, with entries in a ring $R$. 
\item For a prime ideal $\g p$ in a ring $R$, we let $\bm k(\g p)$ stand for the residue field of the local ring $R_{\g p}$. 
\item If $A:V\to V$ is an endomorphism of        vector space over $C$, we let $\mm{Sp}_A$ stand for the set of its eigenvalues. Given $\vr$ an eigenvalue, $\bb G(A,\vr)$ denotes the generalized eigenspace of $A$ associated to $\vr$. 
\item For an abstract  group  or  group scheme $G$, we let $\rep CG$ stand for the category of finite dimensional $C$-linear representations of $G$.
\item Throughout the text,  $\tau$ stands for a subset of $C$ such that the natural map $\tau\to C/\ZZ$ is bijective. 
\item If $A$ and $B$ are subsets of $C$, we denote by $A\ominus B$ the set $\{a-b\,:\,a\in A,\,b\in B\}$. 
\end{enumerate}

\section*{\large{\bf Part I}}
\section{Definitions, terminology and basic results}\label{08.05.2020--1} 

For the convenience of the reader and to ease referencing, we recall some standard definitions.

\begin{dfn} \label{08.04.2020--1}The {\it category of connections},  $\mc(\lau Cx/C)$, has for    
\begin{enumerate}\item[{\it objects}]   those couples $(M,\na)$ consisting of a finite dimensional $\lau Cx$-space and a $C$-linear endomorphism $\na:M\to M$, called  {\it the derivation}, satisfying Leibniz's rule $\na(fm)=\vt(f)m+f\na(m)$, and the  
\item[{\it arrows}] from $(M,\na)$ to  $(M',\na')$ are  $\lau Cx$-linear morphisms $\ph:M\to M'$ such that $\na'\ph=\ph\na$.  
\end{enumerate}

The category of {\it logarithmic connections},    $\mclog(\pos Cx/C)$, has for     
\begin{enumerate}\item[{\it objects}] those couples $(\cm,\na)$ consisting of a finite   $\pos Cx$-module  and a $C$-linear endomorphism, called  {\it the derivation}, $\na:\cm\to\cm$ satisfying Leibniz's rule $\na(fm)=\vt(f)m+f\na(m)$, and   
\item[{\it arrows}] from $(\cm,\na)$ to $(\cm',\na')$ are $\pos Cx$-linear morphisms $\ph:\cm\to \cm'$ such that $\na'\ph=\ph\na$.  
\end{enumerate}
\end{dfn}

As is well-known, $\mc(\lau Cx/C)$ is an abelian category: subobjects, respectively quotients, shall be called subconnections, respectively quotient connections. Also, when speaking of the rank of a connection, we shall mean the dimension of the underlying $\lau Cx$-vector space. 
The category $\mclog(\pos Cx/C)$ is also abelian.

We posses an evident $C$-linear functor 
\[
\ga:\mclog (\pos Cx/C)\aro\mc (\lau Cx/C).
\]

\begin{dfn}\label{15.04.2020--1}An object $M\in\mc(\lau Cx/C)$ is said to be {\it regular-singular} if it   is   isomorphic to a certain $\ga(\cm)$. The full category of $\mc(\lau Cx/C)$ whose objects are regular-singular shall be denoted by $\mcrs(\lau Cx/C)$. 

Given $M\in\mcrs(\lau Cx/C)$, any object 
$\cm\in\mclog(\pos Cx/C)$ such that $\ga(\cm)\simeq M$ is called a {\it logarithmic model of $M$}. In case the model $\cm$ is, in addition, a {\it free} $\pos Cx$-module, we shall speak of a logarithmic {\it lattice}.
\end{dfn}

It is not hard to see that any object in $\mcrs(\lau Cx/C)$ admits a logarithmic lattice; indeed,  if $\cm$ is a logarithmic model, then $\cm_{\rm tors}=\{m\in\cm\,:\,xm=0\}$ is stable under $\vt$ and   $\cm/\cm_{\rm tors}$ is the desired logarithmic lattice.

Given $(\cm,\na)$ and $(\cm',\na')$ in $\mclog(\pos Cx/R)$, the $\pos Cx$--module $\cm\ot_{\pos Cx}\cm'$ becomes a logarithmic connection by means of 
\[
\na\ot\na':\cm\ot\cm'\aro\cm\ot\cm',
\]
\[ \sum m_i\ot m_i'\longmapsto\sum_i\na(m_i)\ot m_i'+m_i\ot\na'(m_i').
\]
We then obtain in   $\mclog(\pos Cx/C)$ 
the structure of a $C$-linear tensor category which gives $\mcrs(\lau Cx/C)$ the structure of a $C$-linear tensor category. (Note that in $\mclog$ we do not always have ``duals.'') Similar constructions then allow us to obtain the next proposition, which is explicitly written down in \cite[Lemma 3.10]{van_der_put-singer03}. See also the proof of  Proposition \ref{25.09.2020--1} further ahead. 

\begin{prp} The category   $\mcrs(\lau Cx/C)$ is an abelian subcategory of $\mc(\lau Cx/C)$ which is stable under direct sums, duals and tensor products. Furthermore, given $(M,\na)\in\mcrs(\lau Cx/C)$ and  a subobject $(M',\na')\subset (M,\na)$, respectively a quotient   $(M,\na)\to(M'',\na'')$, then both $(M',\na')$ and $(M'',\na'')$ are regular-singular. 
 \qed
\end{prp}

Or course, not all objects of $\mclog(\pos Cx/C)$ have ``duals''.

\begin{ex}[Twisted models]\label{09.09.2020--3}Let $\de\in\ZZ$. Write $\one(\de)$ for the $\pos Cx$-submodule of $\lau Cx$ generated by $x^{-\de}$.  Clearly $\vt(\one(\de))\subset\one(\de)$ and in this way, whenever $\de\le0$,  we obtain a subobject of $(\pos Cx,\vt)$. More generally, for any $\cm\in\mclog(\pos Cx/C)$, we obtain a new logarithmic connection  $\cm(\de)$ by defining $\cm(\de)=\one(\de)\ot\cm$.
\end{ex}

\begin{ex}\label{10.09.2020--1}Let $(\cm,\na)$ and $(\cm',\na')$ be objects from $\mclog(\pos Cx/C)$ and on the $\pos Cx$-module $\ch:=\hh{\pos Cx}{\cm}{\cm'}$, let us define 
\[
D:\ch\aro\ch,\quad h\longmapsto \na'\circ h-h\circ\na.
\]
This defines a logarithmic connection called the {\it internal ``Hom}.'' In analogous fashion, we can defined the internal ``Hom'' for two connections. 

By means of the canonical isomorphism 
\[
\hh{\pos Cx}{\cm}{\cm'}\otu{\pos Cx}\lau Cx\simeq \hh{\lau Cx}{\ga\cm}{\ga\cm'}
\] 
we see that the internal ``Hom''   constructed from two regular-singular connections is also regular--singular. 
\end{ex}

\section{Euler connections}\label{11.09.2020--1}
The simplest class of examples of logarithmic connections is given by ``Euler'' connections (the name is inspired by \cite[4.5]{coddington-levinson55}; it is also adopted by \cite[Example 15.9]{ilyashenko-yakovenko08}). In this section, we shall write $\mc$ and $\mclog$ in place of $\mc(\lau Cx/C)$ and $\mclog(\pos Cx/C)$.

\begin{dfn}[Euler connections]\label{12.06.2020--1}
Let $V$ be a finite dimensional vector space over $C$ and $A\in\mm{End}_C(V)$. The Euler logarithmic connection associated to the couple  $(V,A)$ is   defined by the couple $(\pos Cx\ot_CV,D_A)$, where    $D_A(f\ot v)= \vt(f)\ot v+f\ot Av$. Notation: $eul(V,A)$. 
\end{dfn}
Since Euler connections play a prominent role in the theory, let us spend some more time studying them.  

\begin{dfn}\label{12.06.2020--5}
Let $\bb{End}$ be the category whose   
\begin{enumerate}\item[{\it objects}] are couples $(V,A)$ consisting of a finite dimensional $C$-space $V$ and a $C$-linear endomorphism $A:V\to V$,  and whose 
\item[{\it arrows}] from $(V,A)$ and $(V',A')$ are $C$-linear morphisms $\ph:V\to V'$ such that $A'\ph=\ph A$.  
\end{enumerate}
\end{dfn}
Needless to say, letting $\g e=C$ be the 
one dimensional Lie algebra, 
$\bb{End}$ is none other than $\rep C{\g e}$. In particular,   $\bb{End}$ comes    with a canonical structure of  abelian,  $C$-linear {\it tensor} category \cite[I.3.1-2,]{bourbakilie}. (Its unit object is $(C,0)$.) Moreover,  
 for any couple $(V,A)$ and $(V',A')$ in $\bb{End}$, we can produce an ``internal Hom'' $\hh C{(V,A)}{(V',A')}$ \cite[I.3.3, Proposition 3]{bourbakilie} by endowing $\hh C{V}{V'}$   with the endomorphism 
\[
H_{A,A'}: \hh C{V}{V'}\aro\hh CV{V'},     
\quad \ph\longmapsto A'\ph-\ph A. 
\]

With these properties  in sight,  we now have a functor  
\[
eul:\bb{End}\aro\mclog ;
\]
it is obviously $C$-linear, exact and faithful. In addition, $eul$ is a tensor functor (the tensor structure on $\mclog$ is explained in Section \ref{08.05.2020--1}). 

As it should, the obvious morphism of $\pos Cx$-module
\[
eul(\hh C{(V,A)}{(V',A')})\aro \hh{\pos Cx}{\pos Cx\ot V}{\pos Cx\ot V'}
\]
defines an isomorphism in $\mclog$, where the right-hand-side has the ``internal Hom'' logarithmic connection  (cf. Example \ref{10.09.2020--1}). 

We end this section by studying  the influence of $eul$ on Hom sets.

\begin{lem}\label{10.03.2020--1}The following claims are true. 

\begin{enumerate}[1)]\item Suppose that $\mm{Sp}_A$ contains no negative integer. Then any horizontal section of $eul(V,A)$ has the form $1\ot v         $ with $v\in{\rm Ker}(A)$.  

\item Let $(V,A)$ and $(V',A')$ have the following property: the difference $\mm{Sp}_{A'}\ominus\mm{Sp}_{A}$ contains no negative integer. Then each arrow $\Ph:eul(V,A)\to eul(V',A')$ is of the form $\id\ot\ph:\pos Cx\ot_CV\to\pos Cx\ot_CV'$ for a certain $\ph:V\to V'$ such that $A'\ph=\ph A$.  In addition, if $\id\ot\ps=\Ph$, then $\ph=\ps$. Said otherwise, the natural arrow
\[
\hh{{\bf End}}{(V,A)}{(V',A')}\aro \hh{\mclog}{eul(V,A)}{eul(V',A')}
\]
is bijective.
\end{enumerate}\end{lem}
\begin{proof}(1) For each $v\in\mm{Ker}(A)$, the element  $1\ot v\in eul(V,A)$ is clearly horizontal. Conversely, let $\sum_nx^n\ot v_n$ be horizontal. Then 
\[
0=\sum_n x^n\ot (Av_n+nv_n).
\] 
This shows that $v_0\in\mm{Ker}(A)$. In addition,  if   $n>0$,  the equation $Ac=-nc$ cannot have a non-zero solution in $V$, and hence $v_n=0$. 

(2) Let $\Ph:eul(V,A)\to eul(V',A')$ be a non-zero arrow in $\mclog$ and    regard it as a non-zero horizontal element of 
\[
\hh{}{eul(V,A)}{eul(V;,A')}\simeq eul(\hh CV{V'},H_{A,A'}).
\] 
The assumption on the spectra together with a classical result from linear Algebra  shows that  $H_{A,A'}$ cannot have a negative integer as  eigenvalue:  Indeed, if $T\not=0$ is such that $A'T-TA=-k T$, then $\mm{Sp}_{A'+k\mm I}\cap\mm{Sp}_A\not=\varnothing$ \cite[Theorem 4.1, p.19]{wasow76} which forces 
$-k\in\mm{Sp}_{A'}\ominus\mm{Sp}_A$. By part (1), it follows that  $\Ph\in \pos Cx\ot\hh CV{V'}$ comes from an element  $\ph\in\hh CV{V'}$ such that $0=H_{A,A'}(\ph)$.  The fact that $\ph$ is unique  follows from faithfulness of $eul$. 
\end{proof}


\section{Basic results in the theory of regular-singular connections}
\label{08.04.2020--3} 
We shall continue to write $
\mclog$ instead of $\mclog(\pos Cx/C)$ and $\mcrs$ instead of $\mcrs(\lau Cx/C)$.

\subsection{The residue and its applications}

Given $(\cm,\na)\in\mclog$, the very definition of the Leibniz rule assures that $\na(x\cm)\subset x\cm$ so that we obtain, by passage to the quotient, a $C$-linear endomorphism \[{\rm res}(\na):\cm/(x)\aro\cm/(x)\]   called the {\it residue} of $\na$. The set of eigenvalues of $\mm{res}(\na)$ is named the set of {\it exponents} of $\na$ and will be denoted by $\mm{Exp}(\na)$. 

The relevance of the set of exponents is visible through the following  central results. Their proofs are to be found in the classics \cite{coddington-levinson55} or \cite{wasow76}.

\begin{thm}[{\cite[Theorem 1, 4.4, p.119]{coddington-levinson55} or \cite[Theorem 5.1, p.21]{wasow76}}]\label{03.03.2020--1}Let $\cm$ be a free $\pos Cx$-module of finite rank affording a logarithmic connection $\na:\cm\to\cm$  such that no two of its exponents  differ by a \emph{positive integer} (e.g. they all lie in $\tau$). Then $(\cm,\na)\simeq eul( \cm/(x);\mm{res}(\na))$.\qed 
\end{thm}

\begin{thm}[``Shearing'', cf. {\cite[Lemma, 4.4, p.120]{coddington-levinson55} or \cite[17.1]{wasow76}}]\label{02.03.2020--3}
Let $(E,\na_E)$ be an object of $\mcrs$. Then, it is possible to find a logarithmic lattice $(\ce,\na_\ce)$ for $(E,\na_E)$ such such that all exponents of $\na_\ce$ lie in $\tau$. \qed
\end{thm}

\begin{cor}\label{06.07.2020--3}Let  $(M,\na)\in\mcrs$. Then, there exists a finite dimensional vector space $V$ and $A\in\mm{End}_C(V)$ such that 
\begin{enumerate}[(1)]
\item All eigenvalues of $A$ are in $\tau$, and 
\item 
$M\simeq \ga eul(V,A)$. \qed
\end{enumerate} 
\end{cor}

Another relevant feature of regular-singular connections unfolded by  the exponents is: 

\begin{prp}\label{09.04.2019--1}
Let $\phi:E\to F$ be an arrow of $\mcrs(\lau C x/C)$. Let $\ce$ and $\cf$ be models for $E$ and $F$ and assume that $\cf$ is in fact a lattice. We abuse notation and write $\vt$ for all derivations in sight (viz.    $E\to E$, $\ce\to\ce$, etc).  

\begin{enumerate}[(1)]
\item Let $\vr\in\mm{Exp}(\ce)$ and let $s\in \ce$ be such that 
\[(\vt-\vr)^\mu(s)\in x\ce\]
for a certain $\mu\in\NN$. 
Then, for all $k\in\ZZ$, we have
\[
(\vt-(\vr+k))^\mu(x^k\phi(s))=x^{k+1}\phi(\ce). 
\]  
\item Let $\de$ be the largest integer in   $\mm{Exp}(\cf)\ominus\mm{Exp}(\ce)$. Then $x^\de\phi(\ce)\subset \cf$. In particular, adopting the notation of Example \ref{09.09.2020--3}, there exists $\Ph:\ce\to\cf(\de)$ from $\mclog$ such that $\ga\Ph=\phi$.
\item Suppose that $\mm{Exp}(\cf)\ominus\mm{Exp}(\ce)$  contains no positive integer. Then   the natural arrow \[
\hh{}{\ce}{\cf}\aro\hh {}EF
\] 
is bijective. 
\end{enumerate}
\end{prp}
\begin{proof}
\noindent(1)   Using the formula 
\[
[\vt-(\vr+i)]^\mu x^i=x^i(\vt-\vr)^\mu,  
\]
it follows that  
\[
\begin{split}
[\vt-(\vr+k)]^\mu(x^k\phi(s))&  = x^k(\vt-\vr)^\mu(\phi(s)) 
\\
&=x^k\phi[(\vt-\vr)^\mu(s)]
\\&\in x^{k+1}\phi(\ce).
\end{split}
\]

\noindent(2) If $x^\de\phi(\ce)\subset\cf$ we have nothing to do. Let then $k>\de$ be such that $x^k\phi(\ce)\subset \cf$. 
We choose $\vr\in\mm{Exp}(\ce)$ and  $s\in\ce\setminus x\ce$ such that $(\vt-\vr)^\mu(s)\in x\ce$.  
By the previous item,    
\[
(\vt-(\vr+k))^\mu(x^k\phi(s))\in x^{k+1}\phi(\ce)\subset x\cf.
\]
Since $x^k\phi(s)\in\cf$ and $\vr+k$ cannot be an eigenvalue of $\mm{res}_\cf$, it follows that $x^k\phi(s)\in x\cf$, which means that $x^{k-1}\phi(s)\in\cf$ because $\cf$ has no $x$-torsion. 

Let now $\mu_\al$ be the multiplicity of the exponent $\al$ and write
\[
\ce/x\ce\simeq \bigoplus_{\al\in\mm{Exp}} \mm{Ker}({\rm res}_\ce-\al)^{\mu_\al} . 
\]
For any $t\in\ce$, we have  
\[
t= \sum_\al s_\al +xt',
\]
where $(\vt-\al)^{\mu_\al}(s_\al)\in x\ce$ for each $\al$ and $t'\in\ce$. As a consequence, 
$x^{k-1}\phi(s_\al)\in\cf$ and we conclude that 
\[
x^{k-1}\phi(t)\in \cf.
\]
Proceeding by induction, we conclude that $x^\de\phi(\ce)\subset\cf$.

\noindent(3) Follows easily from the previous item and the observation that an arrow $\phi:\ce\to\cf$ which induces $0:E\to F$, must be trivial as $\cf\to F$ is injective.  
\end{proof}
 
Putting together Corollary \ref{06.07.2020--3} and  Proposition \ref{09.04.2019--1}-(3) we arrive at:

\begin{thm}[Deligne-Manin lattices; {\cite[Proposition II.5.4]{deligne70}}]\label{06.07.2020--4} 
Let $M\in\mcrs$ be given. There exists a logarithmic {\it lattice} $\cm$ for $M$ having all its exponents in $\tau$. In addition, if $\cm'\in\mclog$ is another logarithmic lattice for $M$ with all exponents in $\tau$, then there exists a \emph{unique} isomorphism $\ph:\cm\to \cm'$ rendering diagram 
\[
\xymatrix{\ga(\cm)\ar[r]^-{\sim}\ar[dr]_{\ga(\ph)} & M
\\
&\ar[u]_-{\sim}\ga(\cm')}
\] 
commutative. \qed
\end{thm}

\section{Manin's theory revisited}\label{25.06.2021--1}
In \cite{manin65}, Manin gives a classification of  objects in $\mcrs(\lau Cx/C)$ using certain specific models ($M^\xi$ and $M^{(a)}$ in his notation).  We wish to  rewrite his results in the light of Euler connections (Section \ref{11.09.2020--1}), categories, functors and  group schemes. The strategy of this undertaking is to break up the category of regular-singular connections into those which are unipotent and those which are diagonal. 

As before, we write here \[\text{$\mc$,  $\mcrs$  and  $\mclog$}
\]
instead of 
\[\text{$\mc(\lau Cx/C)$, $\mcrs(\lau Cx/C)$ and $\mclog(\pos Cx/C)$.}
\]

\subsection{Jordan blocks}
For each $\la\in C$ and each positive integer $r$, let 
\[
U_{r,\la}=\begin{pmatrix}\la&0&\cdots&\cdots&0
\\
1&\ddots&&\vdots
\\
0&\ddots&\ddots&&\vdots
\\
\vdots&\ddots&\ddots&\vdots
\\
0&\cdots&0&1&\la\end{pmatrix}
\] be the {\it Jordan matrix} of size $r$ and eigenvalue $\la$. Let  $J_r(\la)$ be the object $(C^r,U_{r,\la})$ of $\bb{End}$ and, for a multi-index of positive integers $\bb r=(r_1,\ldots,r_n)$, let 
\[
J_{\bb r}(\la) = J_{r_1}(\la)\op\cdots\op J_{r_n}(\la). 
\]
 With this notation, Jordan's decomposition theorem and Theorem \ref{06.07.2020--3} immediately prove the ensuing result: 

\begin{thm}[{cf. \cite[Theorem 4]{manin65}}]\label{15.04.2020--2}Let $M\in\mcrs$ be given and suppose that $M$ is indecomposable and of dimension $r$. Then $M\simeq\ga(eul J_r(\la))$ for a certain $\la\in\tau$. \qed
\end{thm}

\subsection{Unipotent objects}\label{17.09.2020--5}
In an abelian tensor category (in the sense of \cite[Definition 1.15, p.118]{deligne-milne82}),  an object is {\it unipotent} if it has a filtration whose graded pieces are isomorphic to the unit object (see for example \cite[Definition 1.1.9]{shiho00}). 
Let $\mcrs^u$ and $\bb{End}^u$ be the categories of unipotent objects in $\mcrs$ and $\bb{End}$. According to 
\cite[Proposition 1.2.1,p.521]{shiho00}, both $\mcrs^u$ and $\bb{End}^u$  are {\it abelian}. (This can, of course, be verified directly without much effort.)
Another straightforward exercise is to show that $\mcrs^u$ and $\bb{End}^u$ are tensor subcategories of $\mcrs$ and $\bb{End}$, respectively.

The following simple lemmas  shall be employed below.

\begin{lem}\label{11.09.2020--2}Let $(V,A)\in\bb{End}$ be given. The ensuing conditions are equivalent. 
\begin{enumerate}[1)]\item  $(V,A)$ is unipotent. 
\item   $A$ is nilpotent.
\item The spectrum of $A$ is $\{0\}$.\qed \end{enumerate}
\end{lem}

\begin{lem}\label{11.09.2020--3}Let $E$ be a unipotent object of $\mcrs$ and $\ps: E\to Q$ an epimorphism in $\mcrs$. Then $Q$ is also a unipotent. \qed 
\end{lem}

With this vocabulary at hand, we now have:  
\begin{thm}\label{25.06.2021--4}The functor \[\ga eul:\bb{End}^u\aro\mcrs^u\] is an equivalence. 
\end{thm}

\begin{proof}Let us choose $\tau$ such that $\tau\cap \ZZ=\{0\}$. If  $(V,A)$ and $(W,B)$ are such that $\mm{Sp}_A$ and $\mm {Sp}_B$ are contained in $\tau$, then Lemma \ref{10.03.2020--1}(2) and subsequently Proposition \ref{09.04.2019--1}(3) assure that the natural arrows 
\[
\begin{split}\hh{\bb{End}}{(V,A)}{(W,B)}&\aro \hh{\mclog}{eul(V,A)}{eul(W,B)}
\\
&\aro\hh{\mcrs}{\ga eul(V,A)}{\ga eul(W,B)} 
\end{split}
\] 
are bijections. Because of Lemma \ref{11.09.2020--2} and the choice of $\tau$, this fact proves that  $\ga eul$ is fully faithful when restricted to $\bb{End}^u$. 

 Let 
$M\in\mcrs$ be non-zero,  unipotent  and indecomposable. By Theorem \ref{15.04.2020--2}, there exists $\la\in \tau$  and $r>0$ such that    $M\simeq \ga eul(J_r(\la))$.  Unipotency assures the existence of a non-trivial arrow 
\[
\ph:\ga eul(J_1(0)) \aro \ga eul(J_r(\la)).
\]
From the bijection 
\[
\hh{}{J_1(0)}{J_r(\la)}\aro \hh{}{\ga eul(J_1(0))}{\ga eul(J_r(\la))}
\]
mentioned before, we conclude that $\hh{}{J_1(0)}{J_r(\la)}\not=0$. This shows that $\la=0$ and consequently   $J_r(\la)$ is unipotent in $\bb{End}$. 
Hence $M$ belongs to the essential image of $\ga eul$. In general, we note that any object of $\mcrs^u$ can be decomposed into a direct sum of   indecomposable objects and that these constituents are unipotent because of  Lemma \ref{11.09.2020--3}.  
\end{proof}

The task of describing the category $\mcrs^u$ now benefits from a well-known fact from the theory of algebraic groups.

If $\GG_a=\spc C[t]$ is the additive group scheme, \cite[II.2.2.6(a), p.178]{demazure-gabriel70} explains that  there exists an equivalence   
\[\mm{lev}_1:\rep C{\GG_a}\aro \bb{End}^u
\]
defined by associating to any representation $\rho:\GG_a\to \GL(V)$ the {\it nilpotent} endomorphism 
\[
\log(\rho(1)):V\aro V.
\] (The logarithm of a unipotent endomorphism is defined as usual \cite[II.6.1, p.51]{bourbakilie}.)
We derive:  
\begin{cor}\label{15.09.2020--2}The composition 
\[\rep C{\GG_a}\arou{\mm{lev}_1} \bb{End}^u\arou{\ga eul}\mcrs^u
\]
is an equivalence.\qed \end{cor}

\subsection{Diagonalizable regular-singular connections}\label{17.09.2020--4}
 
\begin{dfn}A connection $(E,\na)$ in $\mcrs$ is diagonalizable if it is the direct sum of one-dimensional regular-singular connections. The full subcategory of all diagonalizable regular-singular connections shall be denoted by  $\mcrs^\ell$.
\end{dfn} 
 
Obviously  $\mcrs^\ell$ is  $C$-linear, stable under tensor products and duals in  $\mcrs$. 
In addition, it is a standard exercise in the theory of representations of associative rings to prove that $\mcrs^\ell$ is an abelian subcategory of $\mc$ (and hence an abelian subcategory of $\mcrs$). Indeed, letting $\cd$ stand for the ring of differential operators,  
$\mc$ is the category of left $\cd$-modules whose dimension over $\lau Cx$ is finite, and the fact that $\mcrs^\ell$ is an abelian subcategory is a straightforward consequence of the study of semi-simple modules made in \cite[VIII.4]{bourbakialgebre}, see in particular Corollary 3 of no.1 on p. 52.

Let now 
\[
\XX=\left\{\begin{array}{c}
\text{   isomorphisms classes}\\
\text{of rank one objects in $\mcrs$}\end{array}\right\}
\]   
and endow $\XX$ with the group structure induced by the tensor product. 
It is not hard to see that 
\[
C\aro \XX,\quad \la\longmapsto\text{isomorphism class of $\ga eul(C,\la)$}
\]
defines an isomorphism 
\begin{equation}\label{25.05.2020--2}
C/\ZZ\arou\sim \XX;
\end{equation}
see \cite[Theorem 3, p120]{manin65}. 
Indeed, let $(\cl,\na)\in\mclog$ be such that $\cl=\pos Cx\po\ell$ is free of rank one. Then, if  $\na(\ell)=a\ell$ and $\ell'=p\ell$ with $p\in\lau Cx^\ti$,  we see that $\na(\ell')=(a+p^{-1}\vt p)\ell'$. The desired result is a consequence of the fact that 
\[
\lau Cx^\ti\aro \pos Cx,\quad b\longmapsto\fr{\vt b}b
\]
establishes an isomorphism of groups  $\lau Cx^\ti\stackrel\sim\to\ZZ+x\pos Cx$.

Write $\mm{Diag}(\XX)$ for the diagonalizable affine group scheme having $\XX$ as group of characters. Said otherwise, 
\[\mm{Diag}(\XX)=\spc C[\XX], \] where  $C[\XX]$ is the group algebra,  cf. \cite[II.1.2.8, 154ff]{demazure-gabriel70} or  \cite[Part I, 2.5]{jantzen87}. As is well-known, the tensor category $\rep C{\mm{Diag}(\XX)}$ can be identified to the tensor category  of $\XX$-graded finite dimensional vector spaces \cite[II.2.2.5, p.177]{demazure-gabriel70}. Hence, from now on, given $V\in\rep C{{\rm Diad}(\XX)}$, we shall write $V_\xi$ for the component of degree $\xi$. 

For each $\xi\in\XX$, let $\hat \xi\in C$ be such that $\hat\xi+\ZZ$ corresponds, under the isomorphism \eqref{25.05.2020--2}, to $\xi$. Then, for each $V\in\rep C{{\rm Diag}(\XX)}$, we put 
\[
\LL(V) = \lau Cx\ot_C V
\]
and endow it with the derivation $D_V$ obtained from 
\[
D_V (1\ot v_\xi) =  \hat\xi\po( 1\ot v_\xi),\qquad v_\xi\in V_\xi.
\]
Obviously, the map $\LL$ gives rise to a  $C$-linear additive functor 
\[
\LL :\rep C{{\rm Diag}(\XX)}\aro\mcrs^\ell.
\]
It is perhaps useful to note that if $C_\xi$ is the $\XX$-graded vector space with a copy of $C$ in degree $\xi$ and zero elsewhere, then $\LL(C_\xi)=\ga eul(C,\hat \xi)$.

\begin{prp}\label{15.09.2020--1} The functor $\LL$ is a $C$-linear tensor equivalence.  
\end{prp} 
\begin{proof}
The only point requiring close examination is the tensor nature of $\LL$. For that, given $\xi,\eta\in\XX$, define $k(\xi,\eta)\in\ZZ$ by
\begin{equation}\label{14.09.2020--1}
\wh{\xi+\eta}=\hat \xi+\hat\eta+k(\xi,\eta).
\end{equation}
Now, let $V$ and $W$ be objects of $\bb{vect}_\XX$ and define an arrow of $\lau Cx$--spaces 
\[
\xymatrix{ \lau Cx\ot_C (V\ot_CW)\ar[rr]^-{\Ph_{VW}} && \left(\lau Cx\ot V\right)\otu{\lau Cx}\left(\lau Cx\ot W\right)
}
\]
by imposing that 
\[
\xymatrix{ 1\ot_C (v_\xi\ot w_\eta)\ar@{|->}[rr] && x^{k(\xi,\eta)}\po[\left(1\ot v_\xi\right)\ot\left(1\ot w_\eta\right)]
}
\]
whenever $v_\xi\in V_\xi$ and $w_\eta\in W_\eta$. Because of eq. \eqref{14.09.2020--1}, $\Ph_{VW}$ is an isomorphism in  $\mc$. Three lengthy but straightforward verifications assure that the couple $(\LL,\Ph)$ is a tensor functor: Indeed, the associativity constraint is a consequence of 
\[
k(\xi,\eta+\ze)+k(\eta,\ze)=k(\xi+\eta,\ze)+k(\xi,\eta),
\]
the commutative constraint of 
\[
k(\xi,\eta)=k(\eta,\xi),
\]
and the identity constraint of 
$\wh0\in\ZZ$
\end{proof}

Using basic cardinal arithmetic, we derive another simple description of     $\XX=C/\ZZ$ which is well--known in case $C=\CC$.

\begin{lem}\label{25.05.2020--3}The abelian groups $C/\ZZ$ and $C^\ti$ are (non-canonically) isomorphic. In particular, $\mathds X$ and $C^\ti$ are isomorphic. 
\end{lem}
\begin{proof}Let $\mu\subset C^\ti$ be the subgroup of roots of unity; it is a divisible group and hence there exists an isomorphism $C^\ti\simeq\mu\op (C^\ti/\mu)$. Similarly, $C/\ZZ\simeq(\QQ/\ZZ)\op(C/\QQ)$. Since $\mu\simeq\QQ/\ZZ$, we only need to show that $C/\QQ\simeq C^\ti/\mu$.

Now, $C/\QQ$ is a $\QQ$-vector space as is $C^\ti/\mu$ and we prove that   any $\QQ$-basis of $C/\QQ$ has the same cardinal as a $\QQ$-basis of $C^\ti/\mu$. 

Following \cite{bourbakiset}, write $\mm{Card}(S)$ to denote the cardinal of a set $S$. We need a simple result which we are unfortunately unable to find in the literature. 

\vspace{.2cm}
\noindent{\it Claim.} For any infinite dimensional $\QQ$-vector space $V$ with basis $B$,     the equality $\mm{Card}(V)=\mm{Card}(\cb)$ holds.

 As $\mm{Card}(\cb)\le\mm{Card}(V)$, we only need to show that $\mm{Card}(\cb)\ge\mm{Card}(V)$. 
Let $\g F$ be the set of finite subsets of $\cb$ and for each $F\in\g F$, write $V_F$ for the vector space generated by $F$. Clearly, 
\[
\mm{Card}(V )\le \mm{Card}(\coprod_F V_F).
\]
Since $\mm{Card}(V_F)=\mm{Card}(\QQ)$ \cite[Corollary 2, III.6.3]{bourbakiset}, and $\mm{Card}(\QQ)\le\mm{Card}(\g F)$, 
we conclude, with the help of \cite[Corollary 3, III.6.3]{bourbakiset}, that 
\[
\mm{Card}(\coprod_F V_F) =\mm{Card}(\g F).
\]  Finally, let $\g F_n$ be the subset of $\g F$ consisting of those subsets with cardinal bounded by $n$. Clearly, 
$\mm{Card}(\cb)^n\ge\mm{Card}(\g F_n)$, which shows that $\mm{Card}(\cb)\ge \mm{Card}(\g F_n)$ \cite[Corollary 4, III.6.3]{bourbakiset} and consequently that $\mm{Card}(\cb)\ge \mm{Card}(\g F)$. The claim is settled. 
\vspace{.2cm}

To end the proof, we note that $\mm{Card}(C/\QQ)\po\mm{Card}(\QQ)=\mm{Card}(C)$ \cite[Proposition 9, III.5.8]{bourbakiset}, and hence $\mm{Card}(C/\QQ)=\mm{Card}(C)$ \cite[Corollary 4, III.6.3]{bourbakiset}. Likewise, $\mm{Card}(C^\ti/\mu)=\mm{Card}(C^\ti)$ so that $\mm{Card}(C^\ti/\mu)=\mm{Card}(C)$. 
\end{proof}

\subsection{The Deligne tensor product}\label{17.09.2020--1}
In order to put the findings of Section \ref{17.09.2020--5} and Section \ref{17.09.2020--4} together---this is  the theme of Section \ref{24.06.2021--1}---we require Deligne's theory of the tensor product of $C$-linear abelian categories, see \cite[Section 5]{deligne90} and \cite{lopez_franco13}. Since   the amount of material necessary to   explain this theory and the pertinent  results is disproportionate to the rest of this text, we shall dedicate    \cite{dos_santos21tp}   to the matter.  On the other hand, for the convenience of the reader, we present a summary. 

  In what follows,  $\kk$ is  any field. Let  $A_1,\ldots,A_n$ and $X$ be $\kk$-linear abelian categories and  write 
\[\bb{Rex}(A_1,\ldots, A_n\,:\ X) \]
for the category of functors \[F:A_1\ti\cdots  A_n\aro  X\]
which are $\kk$-multi-linear and right-exact in each variable. 

\begin{dfn}[{\cite[Section 5]{deligne90}, \cite[Definition 1]{lopez_franco13}}]\label{22.06.2021--1}
Given $\kk$-linear abelian categories $A$ and $  B$, a couple $( P,T)$ consisting of a $\kk$-linear \textbf{abelian} category $ P$ and a functor $T\in \bb{Rex} ( A, B\,:\, P)$ is called a \textbf{ Deligne tensor product} of $ A$ and $ B$ if the following holds. For each $\kk$-linear abelian category $ X$, the functor   
\begin{equation}\label{24.06.2021--2}
\bb{Rex}( P\,:\, X)\aro \bb{Rex}( A, B\,:\, X),\quad F\longmapsto F\circ T,
\end{equation}
is an equivalence. 
\end{dfn}

\begin{ex}\label{24.06.2021--3}Let $G$ and $H$ be   group-schemes over $\kk$.   It then follows that the usual tensor product of vector spaces 
\[
\rep \kk G\ti\rep\kk H  \aro \rep \kk{G\ti H}
\]
is a Deligne tensor product. See \cite{dos_santos21tp}. 
\end{ex}

As argued by \cite[p.208]{lopez_franco13}, the drawback of Definition \ref{22.06.2021--1} is the requirement that $ P$ be abelian, while the properties involved speak solely of {\it right exactness}. 
For that reason, op.cit. employs a weaker version of the tensor product (the Kelly tensor product) and then studies the cases where the Kelly tensor product is a Deligne tensor product. This allows op.cit. to give a complete proof of  Deligne's existence theorem \cite[Proposition 5.13]{deligne90}, see \cite[Proposition 22]{lopez_franco13}, affirming that if $ A$ and $ B$ are {\it categories with length} (cf. \cite{lopez_franco13} page 217 for the definition), then the Deligne tensor product exists. 

The  question concerning the transport of tensor structures  in the theory of the Deligne tensor product is in order. This is dealt with  in  \cite[5.16-17]{deligne90}, but we found that op.cit. has two omissions: first,   the verification of the various functorial commutativity constraints for coherence is   left to the reader in the beginning of   \cite[Proposition 5.17]{deligne90}. Secondly, nowhere in \cite{deligne90}, a discussion is to be found on the monoidal nature of the functors obtained from monoidal functors via the equivalence \eqref{24.06.2021--2}. We explain these matters in more detail. 

Let $A$ and $B$ be $\kk$-linear abelian categories. 
Let $( A,\ot_{ A},\one_{ A})$  and $( B,\ot_{ B},\one_{ B})$ define symmetric monoidal   structures \cite[VII.1]{maclane98} on each one of them,  and  assume that, in addition, we have $\ot_{ A}\in\bb{Rex}(A,A\,:\,A)$ and $\ot_{ B}\in\bb{Rex}(B,B\,:\,B)$. 
On $ A\ti B$, let us introduce an evident structure of symmetric monoidal category: 
\[\ot_{ {AB}}\in\bb{Rex}( A, B, A, B\,:\, A\ti B),
\]
is given by 
\[
(a,b)\ot_{{AB}} (a',b') = (a\ot_{ A} a',b\ot_{ B} b').
\] 

Suppose that $( P,T)$ is a Deligne tensor product for $ A$ and $ B$. As explained in \cite{dos_santos21tp}, we then have an equivalence     
\[
(-)\circ T^{\ti n}:\bb{Rex}(\underbrace{ P,\ldots, P}_{n}\,:\, X)\aro \bb{Rex}(\underbrace{ A, B,\ldots, A, B}_{2n}\,:\, X)
\]
for each $n\ge1$. Letting \[
\ot_P\in\bb{Rex}(P,P\,:\,P)
\]
correspond  to $T\circ\ot_{ A B}$ under eq. \eqref{24.06.2021--2}, we then have a natural isomorphism 
\[
\mu:\ot_P\circ T^2\stackrel\mu\Longrightarrow T\circ\ot_{AB}.
\]
In addition, letting $\one_P=T(\one_A,\one_B)$, it then follows that $(P,\ot_P,\one_P)$ is a symmetric monoidal category and $T$ is a monoidal functor. These details are verified in \cite{dos_santos21tp}. (Needless to say, the difficulty is ensuring {\it coherence} of the monoidal structure.) 

Finally, let $F:A\ti B\to X$ be  any $\kk$-bilinear functor which is right-exact in each variable. Suppose that, giving $A\ti B$ the symmetric monoidal structure explained above, $F$ is monoidal. Then, a functor $\ov F\in\bb{Rex}(A,B\,:\,P)$ corresponding to $F$ under eq. \eqref{24.06.2021--2} is also monoidal.

\subsection{Conclusions}\label{24.06.2021--1}
Let $(\g T,\boxtimes)$ be the Deligne tensor product of $\mcrs^\ell$ and $\mcrs^u$. If
\[
P:  \mcrs^\ell\ti\mcrs^u\aro \mcrs
\] 
is the obvious  tensor product, we obtain through the equivalence \eqref{24.06.2021--2} a right-exact $C$-linear functor 
\[
\ov P:\g T \aro \mcrs 
\]
and a natural isomorphism  
\begin{equation}\label{15.09.2020--3}
\ov P\circ \boxtimes \stackrel\sim{\Longrightarrow}  P 
\end{equation}
in $\bb{Rex}(\mcrs^\ell\,,\,\mcrs^u\,:\,\mcrs)$.   
In addition, 
Section \ref{17.09.2020--1} assures that $\ov P$ is a tensor functor and   \cite[Proposition 5.13(vi), p.148]{deligne90}   that   $\ov P$ is also left exact.

\begin{thm}[The categorical Manin equivalence]\label{categorical_manin}The above defined functor $\ov P$ is an equivalence of $C$-linear abelian tensor categories.    
\end{thm}

\begin{proof} 
{\it Essential surjectivity:} It   suffices to show that any indecomposable $M\in\mcrs$ belongs to the essential image. According to Theorem \ref{15.04.2020--2}, $M\simeq \ga eul(J_r(\la))$. Using that $J_r(\la)\simeq J_1(\la)\ot J_r(0)$, from eq. \eqref{15.09.2020--3} we obtain   
\[
\begin{split}\ga eul(J_r(\la))&\simeq \ga eul(J_1(\la)) \ot \ga eul(J_r(0))
\\
&=\ov P\left[\ga eul(J_1(\la))\boxtimes \ga eul(J_r(0))\right].
\end{split}
\]

{\it Fully faithfulness:} Let    $L,L'\in\mcrs^\ell$ and $U,U'\in\mcrs^u$ so that we have an arrow induced by $P_{(L,U),(L',U')}$:
\begin{equation}\label{20.04.2020--3}
\hh{\mcrs^\ell}{L}{L'}\ot_C\hh{\mcrs^u}{U}{U'}\aro \hh{\mcrs}{L\ot U}{L'\ot U'} .
\end{equation}
 
That   \eqref{20.04.2020--3} is an isomorphism if $L=\one$ and $U'=\one$ is 
easily verified. Indeed, in this case, if $L'\not\simeq\one$, then   $\hh{\mcrs^\ell}{\one}{L'}=0$ and  $\hh{\mcrs}{U}{L'}=0$, which implies that both sides in   \eqref{20.04.2020--3} vanish; if $\al:\one\stackrel\sim\to L'$, then, $\hh{}{\one}{L'}=C\al$, and using the natural isomorphism  
\[\hh{}{\one\ot U}{L'\ot\one}\arou\sim\hh{}{U}{\one}
\]
we may identify   \eqref{20.04.2020--3} with the arrow which maps $\al\ot \ph$ to $\ph$. 
Making use of duals, we conclude that  \eqref{20.04.2020--3} is an isomorphism for all $L$, $L'$, $U$ and $U'$. Employing then  \cite[Proposition 5.13(v)]{deligne90}, 
we conclude that 
\[
\ov P_{L\boxtimes U,L'\boxtimes U'}:
\hh{\g T}{L\boxtimes U}{L'\boxtimes U'}\aro \hh{\mcrs}{\ov P(L\boxtimes U)}{\ov P(L'\boxtimes  U')}
\]
is an isomorphism.  

Let now $\g T_0$ be the full subcategory of $\g T$ whose objects are finite direct sums of objects of the form $L\boxtimes U$. The previous argument shows that $\ov P$ when restricted to $\g T_0$ is fully faithful.

Let now $X$ and $X'$ be arbitrary objects in $\g T$. We can then find two exact sequences,   
\[K\arou\iota Y\arou\pi X\aro0\]
and 
\[
K'\stackrel{\pi'}\longleftarrow Y' \stackrel{\iota'}\longleftarrow X' \longleftarrow0
\]
in which  $Y$, $Y'$, $K$ and $K'$ are in $\g T_0$. 
This is because each element in $\g T$ is the target of an epimorphism from an object of  $\g T_0$, see p. 212 of \cite{lopez_franco13}. 
That of the second exact sequence is a consequence     of the fact that  $\g T$ is the category of representations of a group scheme (Example \ref{24.06.2021--3}),
and hence any object of $\g T$ is the source of a monomorphism to an object of $\g T_0$. 

 Let $a:\ov P(X)\to \ov P(X')$ be given; as $\ov P|_{\g T_0}$ is full,  there exist $c_0:K\to K'$ and $b_0:Y\to Y'$ such that 
\[
\xymatrix{
&\ov P(K)\ar[r]^{\ov P(\iota)}\ar[d]_{\ov P(c_0)} & \ov P(Y)\ar[r]^{\ov P(\pi)}\ar[d]^{\ov P(b_0)}&\ov P(X)\ar[d]^a\ar[r]&0
\\
&\ov P(K')&\ar[l]^{\ov P(\pi')}\ov P(Y')&\ov P(X')\ar[l]^{\ov P(\iota')}&\ar[l]0 
}
\]commutes. 
As $\ov P|_{\g T_0}$ is faithful,  it is the case that 
\[
\xymatrix{K\ar[d]_{c_0}\ar[r]^{\iota} & Y\ar[d]^{b_0}
\\
K'&Y'\ar[l]^{\pi'}}
\]
commutes. Faithfulness of $\ov P|_{\g T_0}$ assures also that $b_0\iota=0$, since $\ov P(b_0\iota)=0$. Then, there exists $d:X\to Y'$ such that $d\pi=b_0$. Since $\ov P(\pi' d\pi)=0$, we can say that $\pi' d\pi=0$. Hence, there exists $a_0:X\to X'$ rendering \[
\xymatrix{
K\ar[d]_{c_0}\ar[r]^\iota & \ar[d]^{b_0}Y\ar[r]^\pi&\ar[d]^{a_0}X\ar[r]&0
\\
K'&\ar[l]^{\pi'} Y'&X'\ar[l]^{\iota'}&0\ar[l]
}
\]
commutative. As $\iota'$ and $\ov P(\iota')$ are  monomorphisms, and  $\pi$ and $\ov P(\pi)$ are  epimorphisms, we see that $\ov P(a_0)=a$ and that $a_0$ is unique with such a property.
\end{proof}

From now on, the following group scheme 
\[
\boxed{\g Z=\mm{Diag}(\XX)\ti\GG_a}
\]
shall play a relevant role.

Translating the equivalences described in Corollary \ref{15.09.2020--2}, in Proposition  \ref{15.09.2020--1}  and Theorem \ref{categorical_manin}, and applying Example \ref{24.06.2021--3}, 
we arrive at: 

\begin{cor}\label{22.04.2020--1}There exists an equivalence of $C$-linear abelian tensor categories 
\[\Ph:
\rep C{\g Z}\aro \mcrs 
\]
having the following properties: 
\begin{enumerate}[1)] 
\item Let $\xi\in\XX$ induce $\chi: \g Z\to \GG_m$ via ${\rm pr}_1:\g Z\to\mm{Diag}(\XX)$. 
Then $\Ph(\chi)$ lies in the class $\xi$.
\item Let   
$\rho: \GG_a\to\GL(V)$
induce $\si:\g Z\to \GL(V)$ via ${\rm pr}_2:\g Z\to \GG_a$.
Then \[\Ph(\si)\simeq \ga eul\left(V,\log\rho(1)\right).\]\qed
\end{enumerate}
\end{cor}

We now set out to  identify $\mm{Diag}(\XX)\ti\GG_a$ with the algebraic envelope of the abstract group $(\ZZ,+)$. Let us recall what this means. 

Given an abstract group $\Ga$, there exists an affine group scheme $\Ga^{\rm aff}$ (over $C$) and an arrow 
\[u:\Ga\aro\Ga^{\rm aff}(C)\]
such that, for any {\it algebraic} group scheme $G$, the natural map 
\[
\hh{}{\Ga^{\rm aff}}{G}\aro\hh{}{\Ga}{G(C)}
\]
\[
\rho\longmapsto \rho(C)\circ u
\]
is bijective. We know of  three ways of constructing $\Ga^{\rm aff}$: by means of the main theorem of Tannakian theory \cite[Theorem 2.11]{deligne-milne82}, by means of   Freyd's adjoint  functor theorem \cite[Theorem 2, V.6]{maclane98}, or by means of Hochschild-Mostow's method \cite[p.1140]{hochschild-mostow69}, \cite[p.72]{abe80}. In case $\Ga=\ZZ$, the construction is   folkloric, but   the only  concrete references we were able to find  were \cite[5.3]{van_der_put-singer03}, which is not really what we want, and \cite[Example 1, p.23]{bass-lubotzky-magid02}, which is imprecise (there is no  need for $\wh\ZZ$ to appear in their conclusion). 

\begin{lem}\label{25.05.2020--5}Let $\al: \XX\arou\sim C^\ti$ be an isomorphism. Define    
\[
f:\ZZ\aro \mm{Diag}(\XX)(C)=\hh{}{\XX}{C^\ti} 
\]
\[
f(k):\xi\longmapsto\al(\xi)^k,
\]
and  $\iota:\ZZ\to\GG_a(C)$ as being the evident inclusion. Then 
\[
(f,\iota):  \ZZ\aro \g Z(C)
\]
is the affine envelope of $\ZZ$.   
In particular, there exists a tensor equivalence of $C$-linear categories 
\[
\Te:\rep C{\g Z } \aro \rep C\ZZ
\]
such that:
\begin{enumerate}[(1)]
\item Let $\xi\in\XX$ induce
$\chi:\g Z\to \GG_m$ via $\g Z\to\mm{Diag}(\XX)$. 
Then $\Te(\chi)$ corresponds to the representation defined by $1\mapsto\al(\xi)\in C^\ti$.
\item Let   
$\rho: \GG_a\to\GL(V)$
induce  $\si:\g Z\to \GL(V)$ via $\g Z\to\GG_a$. 
Then $\Te(\si)$ corresponds to the representation defined by $1\mapsto\rho(1)\in\GL(V)$.
\end{enumerate}
\end{lem}
\begin{proof}Let $\La$ be an abelian group, $h:\ZZ\to  \mm{Diag}(\La)(C)$ a morphism and   $h_1:C[\La]\to C$  the image of $1\in\ZZ$ under $h$. The morphism of abstract groups 
\[
\La\arou {h_1} C^\ti\arou {\al^{-1}}\XX
\]
gives us a   morphism of group schemes $h^\natural:\mm{Diag}(\XX)\to\mm{Diag}(\La)$.
Clearly   \[h^\natural(C)\circ f =h.\]

Let now $U$ be an algebraic unipotent group scheme and 
$h:\ZZ\to U(C)$ a morphism of abstract groups. Using Theorem 8.3 and Exercise 11 of Chapter 9 from \cite{waterhouse79} plus the fact that $\GG_a$ has no non-trivial subgroup schemes,  
there exists a morphism   $g: \GG_a\to U$ such that $g(1)=h(1)$. This of course just means that $g\iota=h$. 

Let now $G$ be any   algebraic group scheme and $\rho:\ZZ\to G(C)$ a morphism. The closure of $\mm{Im}(\rho)$  is   an abelian group scheme \cite[4.3, Theorem]{waterhouse79} and  as such  can be decomposed into a diagonalizable  and a unipotent part \cite[Theorem 9.5]{waterhouse79}. The previous claims then establish what we want.
\end{proof}
Considering an inverse tensor equivalence to $\Te$  \cite[II.4.4]{saavedra72} and employing  Corollary \ref{22.04.2020--1}, we arrive at:

\begin{cor}\label{22.04.2020--2} The $C$-linear abelian tensor category $\mcrs$ is equivalent to $\rep C\ZZ$. More precisely, following Lemma \ref{25.05.2020--3}, let us fix an isomorphism  
\[
\al: \XX \arou\sim C^\ti .
\]
Then,  there exists an equivalence of tensor $C$-linear categories  
\[
\Ps_\al: \rep C\ZZ\aro \mcrs
\]
having the following properties: 
\begin{enumerate}[(1)]
\item If $L\in\rep C\ZZ$ has dimension one and is defined by letting $1\in\ZZ$ act as  $\la\in C^\ti$,  then   $\Ps_\al(L)$ belongs to the class $\al^{-1}(\la)$.
\item If $V\in\rep C\ZZ$ is defined by the unipotent automorphism $u:V\to V$, then 
\[
\Ps_\al(V)\simeq \ga eul\left(V,\log(u)\right).
\]
 \qed
\end{enumerate}
\end{cor}

For the sake of readability, let us state Corollary \ref{22.04.2020--2} ``in the other direction'' and in the   case   where  $C$ is the field of complex numbers, and 
\[
\al(\text{class of $\ga eul(\CC,\la)$}) = e^{2\pi {\rm i}\la}.
\] 
(Consequently, if $\eta_\la:\ZZ\to\CC^\ti$ is defined by $k\mapsto e^{2\pi\mm ik\la}$, then $\Ps_\al(\eta_\la)\simeq\ga eul(\CC,\la)$.)

\begin{cor}\label{16.07.2020--2}Let $\al$ be defined as before. Then, there exists an equivalence of $\CC$-linear tensor categories 
\[
\Om_\al:\mcrs \arou\sim \rep \CC\ZZ
\]
such that to each endomorphism $A:V\to V$, we have 
\[\Om_\al(\ga eul(V,A))=
\begin{array}{c} \text{the representation of $\ZZ$ on $V$}\\  \text{defined by  $k\mapsto e^{2\pi k \mm iA}$.}
\end{array} 
\]
In other words, ``the exponential'' is an inverse to $\Ps_\al$.
\end{cor}
\begin{proof} We  construct $\Om_\al$ as an inverse equivalence to $\Ps_\al$ following the proof of Theorem 1 in \cite[IV.4]{maclane98}. That this inverse equivalence is {\it automatically} a tensor functor is verified by the considerations in \cite[II.4.4]{saavedra72}. 

In what follows, for a given $(V,A)\in\bb{End}$, we shall write $\eta_A$ to mean the representation $k\mapsto e^{2\pi k\mm i A}$ of $\ZZ$.
Let $(V,A)\in\bb{End}$ be given.  We are required to show that   $\Ps_\al(\eta_A)\simeq\ga eul(V,A)$. Assume firstly that $(V,A)$ is indecomposable as an object of $\bb{End}$; this implies in particular that $A$ has a single eigenvalue $\la$. This being so, $A=\la I+N$, with $N$ nilpotent. Then, $e^{2\pi \mm iA}=e^{2\pi \mm i\la}e^{2\pi \mm iN}$ which shows that   $\eta_A\simeq\eta_\la \ot \eta_{N}$. Then, by Corollary \ref{22.04.2020--2},
\[
\Ps_\al(\eta_A)\simeq \underbrace{\Ps_\al(\eta_\la)}_{\simeq\ga eul(\CC,\la)}\ot\underbrace{\Ps_\al(\eta_{N})}_{\simeq\ga eul (V,N)}
\] 
and we conclude that 
\[
\Ps_\al(\eta_A)\simeq \ga eul(V, \la I+N).
\]
The case in which $(V,A)$ is not indecomposable is treated by considering a decomposition into direct sums and we conclude that $\Ps_\al(\eta_A)\simeq\ga eul(V,A)$, as wanted.
\end{proof}

\begin{rmk}It is possible, if the ground field is $\CC$, to obtain Corollary \ref{22.04.2020--2} using the universal Picard-Vessiot extension \cite[10.2, 262ff]{van_der_put-singer03}.  
\end{rmk}

\section{Connections on $\PP\smallsetminus\{0,\infty\}$ after \cite[15.28--36]{deligne87}}\label{01.07.2020--5}

The theory of regular-singular connections over the ring $C[x^\pm]=C[x,x^{-1}]$ works in close analogy with that of $\lau Cx$.  In this section we review it following Deligne. 

Let $\PP$ stand for the projective line obtained by gluing \[\af_0:=\mm{Spec}\,C[x]\quad\text{ and}\quad \af_\infty:=\mm{Spec}\,C[y]
\] 
along 
the open subsets $\mm{Spec}\,C[x^\pm]$ and $\mm{Spec}\,C[y^\pm]$ via the isomorphism $x=y^{-1}$. As suggested by notation, $0\in\PP$ is the point $(x)$ of $\af_0$ and $\infty\in\PP$ the point $(y)$ of $\af_\infty$. Note that    $\vt:C[x]\to C[x]$   can be extended to a global  section of the tangent sheaf, call it also $\vt$, on  $\PP$.

\begin{dfn}\label{01.07.2021--1}\begin{enumerate}[(1)]\item We let $\mc(C[x^\pm]/C)$ be the category whose \begin{enumerate}\item[\it objects] are couples $(M,\na)$ consisting of a $C[x^\pm]$--module of finite type and a $C$--linear endomorphism $\na:M\to M$ satisfying Leibniz's rule $\na(fm)=\vt(f)m+f\na(m)$; 
\item[\it arrows] between $(M,\na)$ and $(M',\na')$ are just $C[x^\pm]$--linear maps $\ph:M\to M'$ satisfying $\na'\ph=\ph\na$.  
\end{enumerate} 
It is called the category of {\it connections on $\PP\smallsetminus\{0,\infty\}$  or on $C[x^\pm]$}.

\item We let $\mclog(\PP/C)$ be the category whose  
\begin{enumerate}\item[\it objects] are couples $(\cm,\na)$ consisting of a coherent $\co_\PP$--module and a $C$--linear endomorphism $\na:\cm\to \cm$ satisfying Leibniz's rule $\na(fm)=\vt(f)m+f\na(m)$ on all open subsets; 
\item[\it arrows] between $(\cm,\na)$ and $(\cm',\na')$ are   $\co_\PP$--linear maps $\ph:\cm\to \cm'$ satisfying $\na'\ph=\ph\na$.  
\end{enumerate} 
It is called the category of {\it logarithmic connections on $\PP$}. 

\item 
We let
\[
\ga_\PP: \mclog(\PP/C)\aro\mc(C[x^\pm]/C)
\]
be the obvious functor. (If convenient we shall write simply $\ga$.) A connection $(M,\na)$ in $\mc(C[x^\pm]/C)$ is  {\it regular-singular} if $\ga_\PP(\cm)\simeq M$ for a certain $\cm\in\mclog(\PP/C)$; in this case, any such $\cm$ is a  {\it logarithmic model} of $M$. In case $\cm$ is in addition a locally free $\co_\PP$-module, we call $\cm$ a  {\it logarithmic lattice}. 

\item 
The full subcategory of $\mc(C[x^\pm]/C)$ having regular-singular connections as objects   is denoted by $\mcrs(C[x^\pm]/C)$.
\end{enumerate}
\end{dfn}

\begin{rmk}\label{15.10.2020--1}
A fundamental result concerning an object $(M,\na)$ from $\mc(C[x^\pm]/C)$ is that $M$ is automatically a {\it free} $C[x^\pm]$--module. (That it is a projective module can be found in \cite[Proposition 8.9]{katz70}, for instance, but we offer a short proof of a more general  fact in Remark \ref{09.09.2020--1} below.) In addition, proceeding as discussed after Definition \ref{15.04.2020--1}, we can always assure the existence of  {\it logarithmic lattices} for objects in $\mcrs(C[x^\pm]/C)$. 
\end{rmk}

\begin{ex}\label{01.07.2020--4}Let $(V,A)\in\bb{End}$ (see Definition \ref{12.06.2020--5}). We let  $eul_\PP(V,A)\in \mclog(\PP/C)$ be the couple $(\co_\PP\ot_CV ,D_A)$, where $D_A(f\ot v)=\vt f\ot v+f\ot Av$ on any open subset of $\PP$. This construction gives rise to a  functor $eul_\PP:\bb{End}\to\mcrs(C[x^\pm]/C)$.
\end{ex}

The canonical inclusions $C[x^\pm]\to\lau Cx$ and $C[x]\to\pos Cx$
produce $C$-linear  exact tensor functors 
\[{\bf r}_0: \mclog(\PP/C)\aro\mclog(\pos Cx/C),\quad \cm\longmapsto \pos Cx\otu {C[x]}\cm(\af_0)
\]
and 
\[
{\bf r}_0:\mc(C[x^\pm]/C)\aro\mc(\lau Cx/C),\quad M\longmapsto \lau Cx\otu {{C[x^\pm]}}M.
\]
It should be noted that  if $eul_\PP(V,A)$ is as in 
Example \ref{01.07.2020--4}, then 
${\bf r}_0( eul_\PP(V,A))$ is simply $eul(V,A)$, as in Definition \ref{12.06.2020--1}. 

In entirely analogous fashion, we have functors ``${\bf r}_\infty$'' with target $\mclog(\pos Cy/C)$ and $\mc(\lau Cy/C)$. Note, on the other hand, that ${\bf r_\infty}(eul_\PP(V,A))$ then corresponds to $eul(V,-A)$ as $\vt(y)=-y$.

The relation between $\mcrs(C[x^\pm]/C)$ and  $\mcrs(\lau Cx/C)$ is given by

\begin{thm}\label{12.06.2020--2}The functor ${\bf r}_0$ induces an  equivalence between categories of regular-singular connections. 
\end{thm}
In \cite[15.28--36]{deligne87}, Deligne offers a proof of this result by constructing an inverse functor and in studying       loc.cit., we obtained the following sequence of thoughts.  (Which is possibly not exactly what Deligne had in mind.) 

\begin{prp}[Regular singular connections are ``Euler'']\label{09.06.2020--1}The functor $\ga eul:\bb{End}\to\mcrs(C[x^\pm]/C)$ is essentially surjective. More precisely, given $(M,\na)\in\mcrs(C[x^\pm]/C)$, there exists $(\g M,A)\in\bb{End}$ and an isomorphism $\ga eul_\PP(\g M,A)\simeq (M,\na)$. In addition, $A$ can be chosen to have no two distinct eigenvalues differing by a positive integer. 
\end{prp}
\begin{proof}
This is mostly spectral theory of the connection operator. Let    $(M,\na)\in\mcrs(C[x^\pm]/C)$ be given.
There exists a  {\it finite} $C[x]$--submodule $\cm$ of $M$ which is invariant under $\na$ and generates $M$ as a $C[x^\pm]$-module. Note that $\cm$ is necessarily free. For each $k\in\ZZ$, we define $\cm^{(k)} = x^k\cm$ to obtain  a {\it decreasing}, {\it separated}  and {\it exhaustive} filtration of $M$.

Given $k<\ell$, let us write $\cm^{(k,\ell)}$ for the quotient $\cm^{(k)}/\cm^{(\ell)}$ (this is a finite dimensional $C$--space) and $\na_{k,\ell}$ for the $C$--linear map induced by $\na$ on it. 
Since multiplication by $x^k$ induces an isomorphism of $C$-spaces $\cm^{(0,1)}\simeq \cm^{(k,k+1)}$, we can show that 
\[
\mm{Sp}(\na_{k,k+1}) = \{k\}\op\mm{Sp}(\na_{0,1}).
\]
From the exact sequence 
\[
0\aro\cm^{(k+1,k+2)}\aro\cm^{(k,k+2)}\aro\cm^{(k,k+1)}\aro0
\]
we derive,  
\[
\begin{split}
\mm{Sp}(\na_{k,k+2}) &= \mm{Sp}(\na_{k+1,k+2})\cup\mm{Sp}(\na_{k,k+1})
\\&=\left(\mm{Sp}(\na_{0,1})\op\{k+1\}\right)\cup\left(\mm{Sp}(\na_{0,1})\op\{k\}\right),
\end{split}
\]
which in all generality gives
\[\begin{split}
\mm{Sp}(\na_{k,\ell})& = \mm{Sp}(\na_{k,k+1})\cup\cdots\cup\mm{Sp}(\na_{\ell-1,\ell})
\\
&=\bigcup_{j=k}^{\ell-1} \mm{Sp}(\na_{0,1})\op \{j\}.
\end{split}
\] 
We now require:

\begin{lem}\label{12.06.2020--3}
\begin{enumerate}
\item The spectral set $\mm{Sp}(\na)$ is contained in $\bigcup_{k\in\ZZ} \{k\}\op \mm{Sp}(\na_{0,1})$ and is invariant under the action of $\ZZ$ on $C$.

\item Let $\vr\in\mm{Sp}(\na)$. Then, there exists a couple of integers $k<\ell$ such that $\bb G(\na,\vr)\subset\cm^{(k)}$ and $\bb G(\na,\vr)\cap\cm^{(\ell)}=(0)$. In particular, $\dim \bb G(\na,\vr)<\infty$.
\end{enumerate} 
\end{lem}
\begin{proof}
(1) Let $\vr\in\mm{Sp}(\na)$ and let $m\in M$ be an eigenvector. Let $k\in\ZZ$ be such that $m\in \cm^{(k)}\smallsetminus \cm^{(k+1)}$. Then,  \[\vr\in\mm{Sp}(\na_{k,k+1})=\mm{Sp}(\na_{0,1})\op\{k\}.
\] 
This shows the inclusion. The final statement is a consequence of the fact  that if $m$ is an eigenvector for the eigenvalue $\vr$, then $x^km$ is an eigenvector for $k+\vr$. 

(2) Consider $I_\vr$ the set of all $k\in\ZZ$ such that $\vr\in\mm{Sp}(\na_{0,1}) \op\{k\}$. Clearly $I_\vr$ is finite; let $\mu=\min I_\vr$ and $\nu=\max I_\vr$. 
For a given $m\in\bb G(\na,\vr)\smallsetminus\{0\}$, there exists  $k\in\ZZ$ such that   $m\subset  \cm^{(k)}\smallsetminus\cm^{(k+1)}$. Hence, $\vr\in\mm{Sp}(\na_{k,k+1})$ so that $k\in I_\vr$. This implies that 
\[
\mu\le k\le\nu.
\]
Hence,  $m\in \cm^{(\mu)}$ while $m\not\in\cm^{(\nu+1)}$.
 
\end{proof}

Finding a logarithmic lattice for $M$ and looking at the space of sections with poles on $0$ and $\infty$, we can construct an increasing and exhaustive filtration of $M$ by {\it finite dimensional } vector spaces  which is, in addition, stable under $\na$.
It then follows that 
\[
M=\bigoplus_{\vr\in \mm{Sp}(\na)}\bb G(\na,\vr). 
\]
Let us now select a finite set $S\subset \mm{Sp}(\na_{0,1})$ such that 
\[
\mm{Sp}(\na) = \bigsqcup_{k\in\ZZ} S\op\{k\}.
\]
Write 
\[
\g M=\bigoplus_{\vr\in S} \bb G(\vr,\na);
\]
this is a finite dimensional space because of Lemma \ref{12.06.2020--3}.
As multiplication by $x^k$ induces isomorphisms 
\[\bb G(\na,\vr)\arou\sim \bb G(\na,\vr+k), 
\]
we have 
\[
M=\bigoplus_{k\in\ZZ} x^k \g M=C[x^\pm]\ot_C\g M.
\]
Let $A:\g M\to\g M$ be the restriction of $\na$. We then see that 
\[
M\simeq \ga eul_\PP(\g M,A).
\] 
In addition, by construction, no two distinct elements of $\mm{Sp}(A)=S$ can differ by a non-zero integer. 
\end{proof}

\begin{proof}[Proof of Theorem \ref{12.06.2020--2}]  
We know that $\bb r_0$ is faithful
since the $C[x^\pm]$-module of any object in $\mc(C[x^\pm]/C)$ is free (Remark \ref{15.10.2020--1}). Essential surjectivity is an immediate consequence of Corollary \ref{06.07.2020--3} and Example \ref{01.07.2020--4}. 
We consider fullness. Let  $(M,\na)$ and $(M',\na')$ in $\mcrs(C[x^\pm]/C)$ be given.
Because of Proposition \ref{09.06.2020--1}, we may assume that \[(M,\na)=(\co_\PP\ot_CV,D_A)\quad\text{and}\quad(M',\na')=(\co_\PP\ot_CV',D_{A'})\] where $A:V\to V$ and $A':V'\to V'$ have no two distinct eigenvalues differing by an  integer. The result is then a consequence of the explicit determination of $\hh{}{eul(V,A)}{eul(V',A')}$ made in  Lemma   \ref{10.03.2020--1} and Proposition \ref{09.04.2019--1}.  
\end{proof}

Let us now express these findings using the notion of exponents. 

\begin{dfn}Let $\cm\in\mclog(\PP/C)$ be given. Define its set of  exponents, $\mm{Exp}(\cm)$, as \[\mm{Exp}(\cm)=\mm{Exp}(\bb r_0 \cm)\cup\mm{Exp}(\bb r_\infty\cm). \]
\end{dfn}

With this definition, we can fix certain preferred logarithmic models. 

\begin{thm}[Deligne-Manin models]\label{06.07.2020--6} Let $M\in\mcrs(C[x^\pm]/C)$. Then, there exists a logarithmic \textbf{lattice}  $\cm\in\mclog(\PP/C)$ for $M$ whose exponents are all on $\tau$. In addition, if $\cm'$ is another logarithmic lattice for $M$ with all exponents on $\tau$, then there exists a unique isomorphism $\ph:\cm\to\cm'$ rendering diagram 
\[
\xymatrix{\ga_\PP(\cm)\ar[r]^-{\sim}\ar[dr]_{\ga_\PP(\ph)} & M
\\
&\ar[u]_-{\sim}\ga_\PP(\cm')}
\] 
commutative. 
\end{thm}

\begin{proof}
There exists, by Corollary \ref{06.07.2020--3}, an object $(V,A)\in\bb{End}$ with $\mm{Sp}_A\subset\tau$ and an isomorphism $u_0:\ga eul(V,A)\stackrel \sim\to\bb r_0(M)$. Let $\bb M_0=eul_\PP(V,A)$; this is an object of $\mclog(\PP/C)$. 
Since $\bb r_0(\ga_\PP(\bb M_0))=\ga eul(V,A)$, Theorem 
\ref{12.06.2020--2} produces an isomorphism $\wt u_0:\ga_\PP(\bb M_0)\stackrel\sim\to M$ such that $\bb r_0(\wt u_0)=u_0$.
Similarly, we obtain $(W,B)\in\bb{End}$, $u_\infty:\ga eul(W,B)\stackrel\sim\to\bb r_\infty(M)$, $\bb M_\infty=eul_\PP(W,B)$ and $\wt u_\infty:\ga_\PP(\bb M_\infty)\stackrel\sim\to M$. From this we derive an isomorphism from $\mc(C[x^\pm]/C)$:
\[
v:\ga_\PP(\bb M_0)\arou\sim \ga_\PP(\bb M_\infty)
\]
and hence an object of $\cm\in\mclog(\PP/C)$ with the properties required in the statement. 

Let now $\cm$ and $\cm'$ be as in the statement; we posses an isomorphism in $\mc(C[x^\pm]/C)$: 
 $f:\ga_\PP(\cm)\stackrel\sim\to\ga_\PP(\cm')$. Write $\cm_0=\cm(\af_0)$ and   $\cm_0'=\cm'(\af_0)$ so that we have an isomorphism of $C[x^\pm]$-modules 
 \[f:\cm_0\otu{C[x]}C[x^\pm]\arou\sim\cm_0'\otu{C[x]}C[x^\pm].\]
 Going over to $\lau Cx$ and using Theorem \ref{06.07.2020--4}, we conclude that \[f\left( \cm_0\otu{C[x]}\pos Cx\right)\subset\cm_0'\otu{C[x]}\pos Cx
\]
and this allows us  extend $f$ to a morphism of $\ph_0 :\cm_0\to\cm_0'$ of $C[x]$-modules. Note that $\ph_0$ is the {\it unique} such extension and that it is  automatically compatible with the derivations; all this because  $\cm_0'\to\cm_0'\ot C[x^\pm]$ is injective. In addition, working with the inverse of $f$, we conclude that $\ph_0$ is an isomorphism. The same reasoning can be applied to $\cm_\infty=\cm(\af_\infty)$ and $\cm'_\infty=\cm'(\af_\infty)$ and the proof is concluded. 
\end{proof}

A difference  between the construction given in Theorem \ref{06.07.2020--6} and that of Section \ref{08.04.2020--3} is that the logarithmic model is not necessarily  of the form $eul_\PP(V,A)$: indeed, these are {\it free} $\co_\PP$-modules. Here is an   illustration.

\begin{ex}Let $C=\CC$, $\tau=\{z\in\CC\,:\,0\le\bb{Re}(z)<1\}$ and  $M=\ga eul_\PP(\CC,\fr12)$. In this case, $\cb=eul_\PP(\CC,\fr12)$ is not what we look for since   $\mm{Exp}(\bb r_\infty\cb)=\{-\fr12\}$. Let us now consider $\cm=\co_\PP(\infty)$, which we understand as being defined by $\cm(\af_0)=\CC[x]\po\bb m_0$ and $\cm(\af_\infty)=\CC[y]\po\bb m_\infty$ subjected to the relation  $\bb m_\infty= x^{-1}\bb m_0$. Now define   $\na|{\af_0}$ by  $\na\bb m_0=\fr12\bb m_0$, so that $\na(\bb m_\infty)=-\fr12\bb m_\infty$ and hence $\mm{Exp}(\cm)=\{\fr12\}$. 
\end{ex}

\section*{\large{\bf Part II}}

We shall now concentrate on the study which gives the title to this paper: regular--singular connections depending on parameters. 

\section{Connections with an action of a ring}
\label{28.06.2021--1}

We fix a commutative $C$-algebra $\La$ whose dimension as a vector space is finite.
The following definition is basic: 

\begin{dfn}\label{28.06.2021--2}Let $\cC$ be a $C$-linear category. We define $\cC_{(\La)}$ as the category whose
\begin{enumerate}\item[{\it objects}] are  couples $(c,\al)$ with $c\in\cC$ and $\al:\La\to\mm{End}(c)$ is a morphism of rings, and an
\item[{\it arrow}] from $(c,\al)$ to $(c',\al')$ is a morphism $\ph:c\to c'$ such that $\al'(\la)\circ\ph=\ph\circ\al(\la)$ for all $\la\in \La$. 
\end{enumerate}To ease terminology, we shall also refer to objects in $\cC_{(\La)}$ as objects of $\cC$ with an action of $\La$ and usually abandon the arrow to the ring of endomorphism from notation. In this case, the endomorphism obtained from $\la
\in \La$ will come with no distinctive graphical symbol. 
\end{dfn}

\begin{dfn}\label{16.10.2020--1}Let $M\in(\modules{\pos Cx})_{(\La)}$. We say that $M$ is \textbf{free in relation to $\La$} if\footnote{In \cite{hai-dos_santos20}, we used the expression ``relatively to'', but after a more careful study, we prefer to write ``in relation to'', as will be done in this paper.} there exists a $\La$-module $V$, an isomorphism of $\pos Cx$-modules $\ps:\pos Cx\ot_CV\to M$ such that, for each $\la\in \La$, $f\in\pos Cx$ and $v\in V$, we have
\[
\ps(f\ot\la v)=\la(\ps(f\ot v)).
\] 
\end{dfn}

\begin{rmk}One easily sees that the canonical arrow $\La\ot_C\pos Cx\to\pos\La x$ is an isomorphism and hence  we may identify $(\modules{\pos Cx})_{(\La)}$ with $\modules{\pos\La x}$. Then, a $\pos Cx$--module with action of $\La$ is free in relation to $\La$ if and only if, as a $\pos \La x$-module, it is of the form $\pos \La x\ot_\La V$ for some $\La$-module $V$. The reason for working with $\pos Cx$-modules with an action of $\La$ instead of with $\pos \La x$-modules is justified by the fact that we wish to rely on the theories of connections over $\lau Cx$ and $\pos Cx$. 
\end{rmk}

Here is the first useful property stemming from the definition: 
\begin{lem}Let $M\in(\modules{\pos Cx})_{(\La)}$ be free in relation to $\La$. Then, $M$ is a free $\pos Cx$-module.\qed 
\end{lem}
Another key property is: 
\begin{lem}Let $M\in(\modules {\pos Cx})_{(\La)}$ be free in relation to $\La$. Then, for each ideal $\g l\subset\La$, the $\pos Cx$-module $M/\g lM$ is also free in relation to $\La$. In particular, $M/\g lM$ is a free $\pos Cx$-module.\qed
\end{lem}

We now begin to apply the definition of objects with an action of $\La$ to categories of connections. 

\begin{ex}The category $\bb{End}_{(\La)}$ consists of couples $(V,A)$ where $V$ is a $\La$-module and $A$ is an endomorphism of $\La$-modules. 
\end{ex}

\begin{ex}\label{02.07.2020--1}
The simplest way of constructing objects in $\mclog(\pos Cx/C)_{(\La)}$ is by means of Euler connections. 
Let $V$ be a finite $\La$-module,  $A:V\to V$ a $C$-linear endomorphism and $eul(V,A)$ the associated Euler connection. 
Assume now that  $A$ is, in addition, $\La$-linear (so that $(V,A)\in
\bb{End}_{(\La)}$). Then,  for each $\la\in \La$, the endomorphism $[\la]:\pos Cx\ot_CV\to\pos Cx\ot_CV$ defined by $[\la](f\ot v)=f\ot \la v$ is horizontal and gives $eul(V,A)$ the structure of an object from   $\mclog(\pos Cx/C)_{(\La)}$. Clearly, $\pos Cx\ot_C V\in(\modules{\pos Cx})_{(\La)}$ is free in relation to $\La$. 
\end{ex}

\begin{thm}[Deligne-Manin lattices]\label{10.03.2020--2}Let  
$M\in\mcrs(\lau Cx/C)_{(\La)}$. There exists a  logarithmic \textbf{lattice} $\cm\in\mclog(\pos Cx/C)$ for $M$ and an action of $\La$ on it  such that:
\begin{enumerate}[(1)]
\item All exponents of $\cm$ lie on $\tau$. 
\item The isomorphism $\ga(\cm)\simeq M$ is compatible with the $\La$-actions. 
\item $\cm$ is free in relation to $\La$. 
\item  In fact, $\cm$ and its $\La$ action can be chosen to be of the form $eul(V,A)$, where $(V,A)\in\bb{End}_{(\La)}$ is an in Example \ref{02.07.2020--1}. 
\end{enumerate}

Finally, if 
\[
\ph:M\aro N
\] 
is an arrow  of $\mcrs(\lau Cx/C)_{(\La)}$ and $\cn\in\mclog(\pos Cx/C)$ is a logarithmic lattice of $N$   affording an action of $\La$ and having properties (1)--(3), then there exists a \textbf{unique} $\wt\ph:\cm\to\cn$ in $\mclog(\pos Cx/C)_{(\La)}$
rendering 
\[
\xymatrix{\cm\ar[r]^{\wt \ph}\ar[d]_{\rm can.}&\cn\ar[d]^{\rm can.}\\M\ar[r]_{\ph}&N}
\]
commutative. 
\end{thm}

\begin{proof}By shearing (Theorem \ref{02.03.2020--3}) there exists a logarithmic lattice $\cm$ of $M$ whose exponents are all on $\tau$.  By Theorem \ref{03.03.2020--1} we can say that $\cm=eul(V,A)$, where $A:V\to V$ is an endomorphism of the finite dimensional $C$-space $V$.  Note that $\mm{Sp}_A\subset\tau$. 

Using Proposition \ref{09.04.2019--1},  the natural morphism 
\[
\mm{End}_{\mclog}( {eul}(V,A))\aro \mm{End}_{\mcrs}(M)
\]
is bijective. Hence, we obtain a morphism of rings $\La\to\mm{End}_{\mclog}(eul(V,A))$; this gives an action of $\La$ on $eul(V,A)$ and   condition (2) is tautologically fulfilled.

In order to show that $eul(V,A)$ is free in relation to $\La$, we remark  that, due to Lemma \ref{10.03.2020--1}(2), for each $\la\in\La$, the arrow  
\[
\la:\pos Cx\ot_C V\aro \pos Cx\ot_C V
\] 
in $\mclog(\pos Cx/C)$ is of the form $1\ot \la$ for an arrow $\la:V\to V$  such that $\la\circ A=A\circ \la$.   We therefore obtain an action of $\La$ on $V$.  We have therefore showed that properties (1)--(4) hold.

Let $N$ and $\cn$ be as in the statement. The existence of an arrow  $\wt\ph:\cm\to\cn$ from $\mclog(\pos Cx/C)$
fitting into the commutative diagram 
\[
\xymatrix{\cm\ar[r]^{\wt \ph}\ar[d]_{\rm can.}&\cn\ar[d]^{\rm can.}\\M\ar[r]_{\ph}&N}
\]
is guaranteed by   Proposition \ref{09.04.2019--1}-(3). (Recall that as   $\pos Cx$-modules,  $\cm$ and $\cn$ are free.) That $\wt\ph$ is unique and  respects the actions of $\La$ is a simple consequence of the fact that $\cn\to N$ is an injection.   
\end{proof}

We end this section by showing that what was said before about formal connections is, modified accordingly, valid for regular-singular connections on $C[x^\pm]$ (considered in Section \ref{01.07.2020--5}). We start with   an immediate consequence of Theorem \ref{12.06.2020--2}. 

\begin{cor}\label{07.07.2020--1}The natural functor 
\[
 \mcrs(C[x^{\pm}]/C)_{(\La)}\aro\mcrs(\lau Cx/C)_{(\La)}
\]
deduced from $\bb r_0$
is an equivalence. \qed
\end{cor}

Before stating the next result, let us put forward the analogue of Definition \ref{16.10.2020--1}. 

\begin{dfn}\label{16.10.2020--2}A coherent $\co_\PP$-module $\cm$ with an action of $\La$ is \emph{locally free in relation to $\La$} if there exists a finite $\La$-module $V$ and an isomorphism 
$\cm(\af_0)\simeq V\ot_C\co(\af_0)$, resp. $\cm(\af_\infty)\simeq V\ot_C\co(\af_\infty)$, such that,  under these isomorphisms, the action of $\La$ is given by means of its action on $V$.   
\end{dfn}

\begin{rmk}Obviously, if $\cm\in\mclog(\PP/\CC)_{(\La)}$ is locally free in relation to $\La$, then it is a locally free $\co_\PP$-module. 
\end{rmk}

\begin{thm} [Deligne-Manin models]\label{10.07.2020--3} Let $M\in\mcrs(C[x^\pm]/C)_{(\La)}$. There exists a logarithmic \textbf{lattice} $\cm\in\mclog(\PP/C)$ endowed with an action of $\La$ such that:
\begin{enumerate}[(1)]\item All exponents of $\cm$ lie on $\tau$. 
\item The canonical isomorphism \[\ga_\PP(\cm)\arou\sim M
\]
is compatible with $\La$-actions.  
\item $\cm$ is locally free in relation to $\La$. 
\end{enumerate}

\end{thm}

\begin{proof}This is much the same as the proof of Theorem 
\ref{06.07.2020--6}, except that we make the following replacements. The use of Corollary \ref{06.07.2020--3} is replaced by that of  Theorem \ref{10.03.2020--2}. The use of Theorem \ref{12.06.2020--2} is replaced by that of Corollary \ref{07.07.2020--1}.
\end{proof}

Note that the statement of Theorem \ref{10.07.2020--3} leaves out the uniqueness properties analogous to the ones in Theorem \ref{10.03.2020--2}. The verification of these occupies the following lines.

Let $\cm\in\mclog(\PP/C)_{(\La)}$ and $\de\in\ZZ$. Let $\cm(\de)$ stand for the logarithmic connection obtained by gluing $x^{-\de}\cm(\af_0)$ and $y^{-\de}\cm(\af_\infty)$ via the isomorphism $x=y^{-1}$. 
In the possession of this definition, we have an analogue of Proposition \ref{09.04.2019--1}:

\begin{prp}\label{12.10.2020--1}
Let $\ph:E\to F$ be an arrow of $\mcrs(C[x^\pm]/C)_{(\La)}$. Let $\ce$ and $\cf$ be logarithmic models for $E$ and $F$ and assume that $\cf$ is in fact a lattice. Let $\de$ be the largest integer in   $\mm{Exp}(\cf)\ominus\mm{Exp}(\ce)$. Then $x^\de\ph(\ce)\subset \cf$. In particular, there exists a unique $\Ph:\ce\to\cf(\de)$ from $\mclog(\PP/C)_{(\La)}$ such that $\ga_\PP(\Ph)=\ph$.
\end{prp}
\begin{proof}Let us write $(-)_0$ and $(-)_\infty$ for sections over $\af_0$ and $\af_\infty$.
Similarly as in the proof of Proposition \ref{09.04.2019--1}, we obtain $x^\de\ph(\ce_0)\subset \cf_0$ and $y^\de\ph(\ce_\infty)\subset \cf_\infty$. As $\cf$ is locally free, we extend $\ph$ to $\Ph:\ce\to\cf(\de)$, an arrow of $\mclog(\PP/C)$. As the restriction $\cf(\de)_0\to F$ and $\cf(\de)_\infty\to F$ are injective, we conclude that $\Ph$ is an arrow of $\mclog(\PP/C)_{(\La)}$.
Obviously $\ga_\PP(\Ph)=\ph$. Injectivity of $\cf(\de)_0\to F$ and $\cf(\de)_\infty\to F$
 again assures that $\Ph$ is unique.\end{proof}

\begin{thm}\label{12.10.2020--2} Let $\ph:M\aro N$ is an arrow  of $\mcrs(C[x^\pm]/C)_{(\Lambda)}$. Let $\cm$ and $\cn$ be logarithmic models  for $M$ and $N$ affording an action of $\La$, and having properties (1)--(3) of Theorem \ref{10.07.2020--3}. Then there exists a \textbf{unique} $\Ph:\cm\to\cn$ in $\mclog(\pos Cx/C)_{(\La)}$ satisfying \[\ga_\PP(\Ph)=\ph.\]
\end{thm}

\begin{proof} This is much the same as the last part of the proof of Theorem \ref{10.03.2020--2}, except that we make the following replacement. The use of Proposition \ref{09.04.2019--1}-(3) is replaced by that of  Proposition \ref{12.10.2020--1}. 
\end{proof}

\section{Formal connections with parameters in a ring: basic  results}\label{28.06.2021--5}

We let $R$ be a complete local noetherian $C$-algebra with residue field $C$ and maximal ideal $\g r$. The $C$-algebras $R/\g r^{k+1}$ shall be abbreviated to $R_k$. We let $\vt$ stand for the $R$-linear derivation  on $\pos Rx$ defined by $\vartheta\sum a_nx^n=\sum a_nnx^n$, as well as its extension to $\lau Rx=\pos Rx[x^{-1}]$. Finally, in developing our arguments, we shall find convenient to identify  $\pos Rx/\g r^{k+1}\pos Rx$ and $\pos {R_k}x$ via the canonical morphism \cite[Theorem 8.11, p.61]{matsumura89}. (Note also that this identification is possible replacing $\g r^{k+1}$ by any given ideal of $R$.)
 
We begin by recycling the definitions appearing in section \ref{08.05.2020--1}.
\begin{dfn}\label{08.05.2021--1}
\begin{enumerate}[(1)]\item We let $\mc(\lau Rx/R)$, the category of \textbf{$R$-linear connections}, be the category whose   
objects are couples $(M,\na)$ consisting of a finite $\lau Rx$-module and an $R$-linear endomorphism $\na:M\to M$ satisfying Leibniz's rule $\na(fm)=\vt(f)m+f\na(m)$, and whose arrows are defined   by imitating  Definition \ref{08.04.2020--1}. 
\item We let  $\mclog(\pos Rx/R)$, the category of \textbf{$R$-linear logarithmic connections}, be the category whose   objects  are couples $(\cm,\na)$ consisting of a  finite   $\pos Rx$-module and an $R$-linear endomorphism $\na:\cm\to\cm$ satisfying Leibniz's rule $\na(fm)=\vt(f)m+f\na(m)$, and whose arrows are defined   by imitating  Definition \ref{08.04.2020--1}. Whenever no confusion is possible, we omit reference to $\na$ in the notation. 
\item We denote by   
\[
\ga:\mclog(\pos Rx/R)\aro\mc(\lau Rx/R)
\]
the obvious functor and define $\mcrs(\lau Rx/R)$,  the category of \textbf{regular-singular} connections, as being the full subcategory of $\mc(\lau Rx/R)$ whose objects are (isomorphic to an object) in the image of $\ga$. 
\item 
Given $M\in\mcrs(\lau Rx/R)$, any object  $\cm\in\mclog(\pos Rx/R)$
for which there is an isomorphism $\ga(\cm)\simeq M$ is said to be a \textbf{ logarithmic model} of $M$.
\item A logarithmic model $\cm$ of $M$ is  called   \textbf{$x$-pure} if multiplication by $x$ is injective on $\cm$.
\end{enumerate} 
\end{dfn}

It comes as no surprise that $\mc(\lau Rx/R)$ is an abelian category such that the forgetful functor to $\modules{\lau Rx}$ is exact.

Given $(\cm,\na)\in\mclog(\pos Rx/R)$, 
it is clear that the $\pos Rx$-module $\bigcup_k(0:x^k)_\cm=\{m\in \cm\,:\,x^km=0\}$ is stable under $
\na$, so that, taking the quotient, we have: 

\begin{lem}\label{15.07.2020--1}Each $M\in\mcrs(\lau Rx/R)$
has an $x$-pure logarithmic model. \qed
\end{lem}

This simple result can be improved, see Theorem \ref{16.05.2020--1} below. 
But its utility is promptly manifest. 

\begin{prp}\label{25.09.2020--1}The full subcategory $\mcrs(\lau Rx/R)$ of $\mc(\lau Rx/R)$ is stable under quotients and subobjects. 
\end{prp}
\begin{proof}[Sketch of proof.] Let  $N\in\mc(\lau Rx/R)$ be a subobject of $M\in\mcrs(\lau Rx/R)$.  Let $\cm$ be an $x$-{\it pure} logarithmic model for $M$ (cf. Lemma \ref{15.07.2020--1}). Then $\cn:=\cm\cap N$ is an $x$-pure logarithmic model of $N$. Quotients are treated using models for the kernel.   
\end{proof}

Furthermore, given $(\cm,\na)$ and $(\cm',\na')$ in $\mclog(\pos Rx/R)$, their tensor product $\cm\ot_{\pos Rx}\cm'$   gives rise to an object of $\mclog(\pos Rx/R)$ by decreeing 
\[
\na\ot\na'(m\ot m') = \na (m)\ot m'+m\ot \na'(m').
\]
It is then the case that    $\mclog(\pos Rx/R)$ becomes an $R$-linear tensor category and $\mcrs(\lau Rx/R)$ is an $R$-linear abelian tensor category.  

\begin{ex}[Twisted models]\label{24.09.2020--1} For each $\de\in\ZZ$, let  $\one(\de)$ denote the {\it free} $\pos Rx$-submodule of $\lau Rx$ generated by $x^{-\de}$. Clearly, $\one(\de)$ is invariant under $\vt$ and we obtain in this way an $x$-pure logarithmic model for the trivial object $(\lau Rx,\vt)$. We define analogously, for each $\cm\in\mclog(\pos Rx/R)$ the object $\cm(\de)$ as being $\one(\de)\ot\cm$. 
\end{ex}  
We now explore further immediate similarities of this theory and the classical one.

\begin{ex}Let $\bb{End}_R$ be the category whose objects are couples $(V,A)$ consisting of a finite $R$--module $V$ and an $R$--linear endomorphism $A:V\to V$, and whose arrows are given as in Definition \ref{12.06.2020--5}. 
Given $(V,A)\in\bb{End}_R$, let $D_A:\pos Rx\ot_RV\to\pos Rx\ot_RV$ be defined by 
\[
D_A(f\ot v)=\vt f\ot v+f\ot Av.
\]
This gives rise to an $R$-linear functor 
\[
eul :\bb{End}_R\aro \mclog(\pos Rx/R)
\]
analogous to the one in Definition \ref{12.06.2020--1}. 
\end{ex}

Let $(\cm,\na)\in\mclog(\pos Rx/R)$ and note that 
\begin{equation}\label{16.05.2020--2}
\mm{res}_\na: \cm/(x)\aro\cm/(x),
\end{equation}
given by 
\begin{equation}\label{16.05.2020--3}
\mm{res}_\na(m+(x) ) = \na(m)+(x), 
\end{equation}
 is $R$-linear.

\begin{dfn}[Residue and exponents]\label{28.06.2021--6}
The $R$-linear map \eqref{16.05.2020--2} is called the {\it residue} of $\na$. If 
\[\ov{\mm{res}}_\na : \cm/(\g r,x)\aro\cm/(\g r,x)
\] 
stands for the $C$-linear morphism obtained from $\mm{res}_\na$ by reduction modulo $\g r$, we call the set $\mm{Sp}_{\ov{\mm{res}}_\na}$ the set of {\it exponents} of $\na$; it shall be denoted by $\mm{Exp}(\cm,\na)$, $\mm{Exp}(\na)$ or $\mm{Exp}(\cm)$ if no confusion is likely. 
\end{dfn}

\begin{rmk}\label{09.05.2022--1}
It should be highlighted that the {\it exponents belong to $C$}. 
The reason for taking this path is, from a practical viewpoint, justified by the fact that we are able to prove the results we wanted 
with it. But it is important to throw more light on our choice. While explaining either this work or \cite{hai-dos_santos20} to others, the question ``why not take, in case $R$ is a domain, the exponents in a quotient field of $R$?'' frequently  appeared. This is certainly a possible path and when we started this theory, our exponents (in Definition \ref{28.06.2021--6}) were called  {\it reduced exponents}. Then, at some point it became clear that: (a) Reduced exponents were the ones controlling the theory and leading to Corollary \ref{16.07.2020--1},  our main result; (b) In taking limits, we need no-reduced rings; (c) In taking limits, it is important to have the   exponents being constant while ``reducing'', see Corollary \ref{17.05.2020--1}. We then decided that the reduced exponents deserved a prominent name. 
On the other hand, in different situations, our definition may   be insufficient; see Remark \ref{09.05.2022--2} below.  
\end{rmk}

Let us now start by recalling

\begin{lem}[{\cite[II, Problem 4.1]{wasow76}}]\label{01.07.2019--2} Let 
$m$ and $n$ be positive integers,  $A$  an element of $\mm M_{m}(C)$, and $B$ an element of $\mm M_{n}(C)$. Let   \[
f:\mm M_{m\ti n}(C)\aro \mm M_{m\ti n}(C)
\]
be the linear map defined by $X\mapsto AX-XB$. Then $\mm {Sp}_f=\mm{Sp}_A\ominus\mm{Sp}_B$.   In particular, if 
no two distinct eigenvalues of $A$ differ by an integer, then  the linear transformation $\nu\id-\mm{ad}_A:\mm M_m(C)\to\mm M_m(C)$ is invertible for each $\nu\in\ZZ\smallsetminus\{0\}$. \qed
\end{lem}

A direct application of Lemma \ref{01.07.2019--2} and Nakayama's Lemma shows the following.
\begin{cor}\label{01.07.2019--3}   Let $A\in\mm M_n(R)$ be such that its reduction modulo $\g r$, call it  $\ov A\in \mm M_n(C)$, has no two distinct eigenvalues differing by an integer. Then, for any  $\nu\in\ZZ\smallsetminus\{0\}$, the $R$-linear morphism $\nu\id-\mm{ad}_A:\mm M_n(R)\to\mm M_n(R)$ is bijective. \qed
\end{cor}

\begin{thm}[compare to Theorem \ref{03.03.2020--1}]\label{16.05.2020--5}Let $(\cm,\na)\in\mclog(\pos Rx/R)$ be such that $\cm$ is a free $\pos Rx$-module and  no two distinct exponents of $\na$ differ by an  integer. 
Then, $(\cm,\na)$ is isomorphic to $eul(\cm/(x),\mm{res}_\na)$.

Otherwise said, consider a differential system 
\[
\vt\bm y=A\bm y
\]
defined by  $A\in\mm M_r(\pos Rx)$ such that $A(0)$ modulo $\g r$ has no two distinct eigenvalues   differing by an integer. Then, there exists $P\in\GL_r(\pos Rx)$ such that, writing $\bm y=P\bm z$, we arrive at the system 
\[
\vt\bm z= B\bm z
\]
in which   $B\in\mm M_r(R)$. 
\end{thm}
\begin{proof}One proceeds as in Sections 4.2, 4.3 and 5.1 of \cite{wasow76}, but substitute the use of  Wasow's Theorem 4.1 by our Corollary \ref{01.07.2019--3}. 
\end{proof}

We now move to shearing
techniques which   allows us to eliminate the hypothesis on the exponents in Theorem \ref{16.05.2020--5}. We begin by setting up the necessary linear algebra.

\begin{prp}\label{10.07.2019--1}Let $\La$ be a commutative $C$-algebra which is a finite dimensional  $C$-space. Let $\g n\subset\La$ be a nilpotent ideal,     $V$ a finite $\La$-module and $A:V\to V$  a  $\La$-linear arrow. Considering $A$ as a $C$-linear endomorphism, write $\vr_1,\ldots,\vr_r$ for its  distinct eigenvalues and let  
\[
V=\bb G(A,\vr_1)\op\cdots \op \bb G(A,\vr_r)
\]
be the decomposition into generalized eigenspaces. 

\begin{enumerate}[(1)]\item Each $\bb G(A,\vr_j)$ is invariant under $\La$ and $\g n\po\bb G(A,\vr_j)\not=\bb G(A,\vr_j)$. 

\item Write 
$\ov V=V/\g n V$ and  $\ov {\bb G(A,\vr_j)}$ for the image of $\bb G(A,\vr_j)$ in $\ov V$. Then 
$\ov {\bb G(A,\vr_j)}\not=0$. 
\item Let $\ov A$ be the endomorphism of $\ov V$ induced by $A$. Then the space $\ov {\bb G(A,\vr_j)}$ is the generalized eigenspace of $\ov A$ associated to $\vr_j$ and $\mm{Sp}_{\ov A}=\mm{Sp}_A$.
\end{enumerate}
\end{prp} 
\begin{proof}
 By definition,  
\[
\bb G(A,\vr_j)  =\bigcup_n\, \mm{Ker}\,(A-\vr_j\id)^{n},
\]  
so that for every $\la\in\La$, we have      $\la \bb G(A,\vr_j)\subset \bb G(A,\vr_j)$. 
 Since  $\bb G(A,\vr_j)\not=0$, we know that $\g n\po \bb G(A,\vr_j)\not=\bb G(A,\vr_j)$. This establishes (1).  To prove (2), we note that $\g nV=\oplus_j \g n\bb G(A,\vr_j)$ and hence $\bb G(A,\vr_j)/\g n\bb G(A,\vr_j)\stackrel \sim\to\ov{\bb G(A,\vr_j)}$. Also, as a consequence, we  arrive at the direct sum decomposition
\begin{equation}\label{23.09.2020--1}
\ov V=\ov{\bb G(A,\vr_1) } \op\cdots\op \ov{\bb G(A,\vr_r)}.
\end{equation}
 Nilpotence of  $\ov A-\vr_j\id$ when restricted to $\ov{\bb G(A,\vr_j)}$ 
now shows that  $\vr_j$ is the only eigenvalue of $\ov A $ on $\ov{\bb G(A,\vr_j)}$ and   that
\[
\ov {\bb G(A,\vr_j)}\subset  \bb{G}(\ov A,\vr_j).
\]
   Let us fix $j_0\in\{1,\ldots,r\}$ and show that  $\ov{\bb G(A,\vr_{j_0})}\supset\bb{G}(\ov A,\vr_{j_0})$. Suppose that 
 $\ov w\in\ov V$ is annihilated by $(\ov A-\vr_{j_0}\id)^m$ and write it as $\ov v_1+\cdots+\ov v_r$ with $\ov v_j\in\ov{\bb G(A,\vr_j)}$. Since $\ov v_j\in\ov{  \bb G(A,\vr_j)}$, there exists  $n_j\in\NN$ such that $(\ov A-\vr_j\id)^{n_j}(\ov v_j)=0$. 
 We now chose   $\mu=\max\{m ,n_1,\ldots,n_r\}$ and then  find   $P,Q\in C[T]$ such that 
\[
P(T)\po(T-\vr_{j_0})^\mu=1+Q(T)\po\prod_{j\not={j_0}}(T-\vr_j)^\mu. 
\]
Hence, 
\[0=
 \ov w+Q(\ov A)\po \prod_{j\not={j_0}}(\ov A-\vr_j\id)^\mu(\ov w).
\]
Now 
\[Q(\ov A)\po \prod_{j\not={j_0}}(\ov A-\vr_j\id)^\mu( \ov w )=\underbrace{Q(\ov A)\po \prod_{j\not={j_0}}(\ov A-\vr_j\id)^\mu(\ov v_{j_0})}_{\in\ov  {\bb G(A,\vr_{j_0})} },
\]which shows $\ov w\in \ov  {\bb G(A,\vr_{j_0})}$. Finally, \eqref{23.09.2020--1} is the decomposition of $\ov V$ into generalized eigenspaces. 
\end{proof}

The previous result also allows us to grasp the utility of our definition of exponents. 

\begin{cor}\label{17.05.2020--1}
\begin{enumerate}
\item Let $\La$ be a $C$-algebra which is a finite dimensional vector space and $\g n\subset \La$ a nilpotent ideal. Let $(\cm,\na)\in\mclog(\pos Cx/C)_{(\La)}$ and define $\cm|_{\g n}=\cm/\g n$. Then $\na$ gives rise to $\na|_{\g n}:\cm|_{\g n}\to \cm|_{\g n}$ and the couple $(\cm|_{\g n},\na|_{\g n})$ is an object of $\mclog(\pos Cx/C)_{(\La/\g n)}$ which has the same set of exponents as $(\cm,\na)$.
\item 
Let $(\cm,\na)\in\mclog(\pos Rx/R)$ and $k\in\NN$ be given. 
Define $\cm|_k:=\cm/{\g r}^{k+1}$. Then, this is a $\pos Cx$-module of finite type (since it is a finite $\pos {R_k}x$--module). 
Let $\na|_k:\cm|_k\to\cm|_k$ be induced by $\na$. Then $(\cm|_k,\na|_k)$ is an object of $\mclog(\pos Cx/C)_{(R_k)}$ and $\mm{Exp}(\na)=\mm{Exp}(\na|_k)$.
\end{enumerate}  \qed
\end{cor}

Another useful consequence of Proposition \ref{10.07.2019--1} is
\begin{cor}[Lifting of Jordan decomposition]\label{10.07.2019--2}Let $V$ be an  $R$-module and   $A:V\to V$ be an $R$-endomorphism. Denote by $\ov A:\ov V\to \ov V$ the $C$-linear endomorphism obtained by reducing $A$ modulo $\g r$. 

Then, there exist   $R$-submodules $\{V(\vr),\,:\,\vr\in
\mm{Sp}_{\ov A}\}$ of $V$ enjoying the following properties: 
\begin{enumerate}[(1)]\item The $R$-module $V$ is the  direct sum of $\{V(\vr)\,:\,\vr\in\mm{Sp}_{\ov A}\}$. 
\item Each $V(\vr)$ is stable under $A$.
\item If $\ov{V(\vr)}$ stands for the image of $V(\vr)$ in $\ov V=V/\g rV$,   then $\ov {V(\vr)}=\bb G( \ov A,\vr)$.
\end{enumerate}
In addition, if $V$ is free, then each $V(\vr)$ is also free. 
\end{cor}

\begin{proof}Let  $\vr$ be fixed. For a given $k\in\NN$, let $A_k:V_k\to V_k$ be the $R_k$-linear endomorphism obtained by reducing $A$ modulo $\g r^{k+1}$. The eigenvalues of $A_k$ shall always mean those of the associated $C$-linear endomorphism of $V_k$. Applying Proposition \ref{10.07.2019--1} to the case $\La=R_{k+1}$ and $\g n= \g r^{k+1}\po R_{k+1}$, we obtain  that $\mm{Sp}_{A_{k+1}}=\mm{Sp}_{A_{k}}$. By induction, $\mm{Sp}_{A_{k}}=\mm{Sp}_{A_0}=\mm{Sp}_{\ov A}$.  
In addition, we also know that the canonical arrow 
\[
\bb G(A_{k+1},\vr)\aro \bb G(A_k,\vr)
\]
is a surjective morphism of $R_{k+1}$-modules whose kernel is $\g r^{k+1}\bb G(A_{k+1},\vr)$. 
Now, we define 
\[
V(\vr) = \lip_k\bb G(A_k,\vr)
\]
which is considered as an $R=\lip_kR_k$-module. According to Proposition \ega{0}{I}{Proposition 7.2.9, p.65}, the natural projection $V(\vr)\to\bb G(A_k,\vr)$ is surjective and has kernel $\g r^{k+1}V(\vr)$. 

Using the inclusions $\bb G(A_k,\vr)\to V_k$, we obtain an injective arrow of $R$-modules 
\[
u:\bigoplus_\vr V(\vr) \aro  \lip_kV_k(\simeq V). 
\]
In addition, reducing $u$ modulo $\g r$ and employing the fact that $V(\vr)/\g rV(\vr)\simeq \bb G(A_0,\vr)$, Nakayama's Lemma \cite[Theorem 2.2, p.8]{matsumura89} tells us that $u$ is surjective. 

The verification of the final assertion is clear: Because $V(\vr)$ is a direct summand of $V$, we can infer that $V(\vr)$ is projective and of finite type, hence free. 
\end{proof}

In case the module $V$ appearing in the statement of Corollary \ref{10.07.2019--2} is     {\it free}, we have the (probably well-known) consequence:

\begin{cor}\label{06.07.2020--1}Let $A\in\mm M_n(R)$ be given and denote by $\{\vr_1,\ldots,\vr_r\}$ the spectrum of $\ov A\in\mm M_n(C)$. Then, there exists
\begin{enumerate}[(1)]
\item $P\in \GL_n(R)$, 
\item a partition $n=n_1+\cdots+n_r$ and 
\item matrices 
\[
U(1)\in\mm M_{n_{1}}(R),\ldots, U(n_r)\in \mm M_{n_r}(R)
\]
\end{enumerate}
such that 
\[
P^{-1}AP=\left(\begin{array}{c|c|c}U(1) & 0& 0\\\hline0&\ddots&0
\\
\hline0&0&U(n_r)
\end{array}\right),
\]
and, for every $i$, the image of $U(n_i)$ in $\mm M_{n_i}(C)$ is a generalized Jordan matrix with eigenvalue $\vr_i$.\qed
\end{cor}

\begin{rmks}\begin{enumerate}[(a)] \item Corollary \ref{06.07.2020--1} should be compared with Theorem 25.1 in \cite{wasow76}. In fact, it is not difficult to show that this result holds under the weaker assumption that $R$ is only strictly Henselian. Indeed, the Hensel property allows us to lift the factorization of the characteristic polynomial of $A$ and one proceeds by showing that the kernels of the various factors evaluated at $A$ produce a direct sum decomposition.   
\item There is a substancial literature on the problem of similarity of matrices over rings, see e.g. \cite{guralnick81} and references in there. 
\end{enumerate}
\end{rmks}

Once in possession of these properties, we can follow the shearing technique in \cite{wasow76} to prove: 

\begin{thm}\label{25.05.2020--1} Let $(\cm,\na_\cm)\in\mclog(\pos Rx/R)$ be such that $\cm$ is a \textbf{free} $\pos Rx$-module and let $(M,\na_M)$ be the regular-singular connection associated to $(\cm,\na_\cm)$. Then, there exists
an object $(W,B)\in\bb{Eul}_R$, with $W$ a free $R$-module,  such that $(M,\na_M)\simeq\ga eul(W,B)$. In addition,  the eigenvalues of the endomorphism of $W/(x,\g r)$ defined by $B$ belong all to $\tau$.

Otherwise said, consider a differential system 
\[
\vt\bm y=A\bm y
\]
defined by  $A\in\mm M_r(\pos Rx)$. There exists $P\in\GL_r(\lau Rx)$ such that, writing $\bm y=P\bm z$, we arrive at the system 
\[
\vt\bm z= B\bm z
\]
in which   $B$ belongs to $\mm M_r(R)$ and  its  image  in $\mm M_r(C)$ only has eigenvalues lying in $\tau$. 
\end{thm}

\begin{proof}
Because of Nakayama's Lemma \cite[Theorem 2.2, p.8]{matsumura89} (and the fact that $\pos Rx$ is local)  a set of elements of $\cm$ which is mapped to a basis of $\cm/(x)$ is necessarily a basis of $\cm$.  
According to Corollary \ref{06.07.2020--1}, there exists a basis $\bm m=\{m_i\}_{i=1}^r$ of $\cm$ such that the basis 
\[
\ov{\bm m}=\{m_i+(x) \}_{i=1}^r\]
 of $\cm/(x)$ has the following properties:
\begin{enumerate}[(a)]
\item the matrix of $\mm{res}_\cm :\cm/(x)\to\cm/(x)$ with respect to $\bm m$ has the form 
\[
\left(\begin{array}{c|c}J_{11} & 0\\\hline0&J_{22}
\end{array}\right)
\] 
where $J_{11}  \in\mm{M}_{q}(R)$ and $J_{22}\in\mm{M}_{r-q}(R)$. (Here $q\in\{1,\ldots,r\}$ is a positive integer. In case  $q=r$, we only say that $\mm{res}_\cm=J_{11}$.)

\item If $\ov J_{11}\in\mm M_q(C)$ and $\ov J_{22}\in\mm M_{r-q}(C)$ stand for the images of $J_{11}$ and $J_{22}$ respectively, then     $\mm{Sp}_{\ov J_{11}}=\{\vr\}$ and    $\vr\not\in\mm{Sp}_{\ov J_{22}}$.
\end{enumerate}

 Hence, the matrix   of $\na_\cm$ with respect to     $\bm m$ 
is 
\[
\left(\begin{array}{c|c}J_{11}+ x\Ps_{11} & x\Ps_{12}\\\hline x\Ps_{21}&J_{22}+x\Ps_{22}
\end{array}\right),
\] 
where $\Ps_{11}\in\mm M_q(\pos Rx)$ and $\Ps_{22}\in\mm M_{r-q}(\pos Rx)$. 

Let us now define $\bm m'=\{m_1',\ldots,m_r'\}\subset M$ by
\[
m'_j=\left\{
\begin{array}{ll} xm_j,&\text{if $j\in\{1,\ldots,q\}$},\\
m_j,&\text{if $j\in\{q+1,\ldots,r\}$.}
\end{array}\right.
\]
which is to say that the base-change matrix from $\bm m$ to $\bm m'$ is 
\[\left(\begin{array}{c|c} x  & 0\\\hline 0& I
\end{array}\right).
\]
Clearly, 
\[\cm'=\sum_{j=1}^r \pos Rx\po m'_j\]
is a free $\pos Rx$-module such that $\cm'[1/x]=M$. In addition, the matrix of $\na_M$ with respect to $\bm m'$ is 
\[
\left(\begin{array}{c|c}1/x&0\\\hline0&I
\end{array}\right)\po \left(\begin{array}{c|c}x&0\\\hline0&0
\end{array}\right)+\left(\begin{array}{c|c}1/x&0\\\hline0&I
\end{array}\right)\po \left(\begin{array}{c|c}J_{11}+ x\Ps_{11} & x\Ps_{12}\\\hline x\Ps_{21}&J_{22}+x\Ps_{22}
\end{array}\right)\po\left(\begin{array}{c|c} x&0\\\hline0&I
\end{array}\right) ,
\]
which equals
\[
\left(\begin{array}{c|c}I&0\\\hline0&0
\end{array}\right)+\left(\begin{array}{c|c}J_{11} + x\Ps_{11} &  \Ps_{12}\\\hline x^2\Ps_{21}&J_{22}+x\Ps_{22}
\end{array}\right).
\]
Hence, with respect to the basis $\{m'_j+(x)\}$ of $\cm'/(x)$, we have 
\[
\mm{res}_{\cm'} = \left(\begin{array}{c|c} J_{11}+I  &  \Ps_{12}  \\\hline 0  & J_{22} 
\end{array}\right)
\]
and the exponents of $\cm'$ are $\{\vr+1\}\cup\mm{Sp}_{\ov J_{22}}$. Analogously, if we define 
   $\bm m''=\{m_1'',\ldots,m_r''\}\subset M$ by
\[
m''_j=\left\{
\begin{array}{ll} x^{-1}m_j,&\text{if $j\in\{1,\ldots,q\}$},\\
m_j,&\text{if $j\in\{q+1,\ldots,r\}$,}
\end{array}\right.
\]
and 
\[
\cm''=\sum_{j=1}^r\pos Rx\po m_j'',
\]
we obtain a logarithmic model $(\cm'',\na_M)$ such that $\mm{Exp}_{\cm''}=\{\vr-1\}\cup\mm{Sp}_{
\ov J_{22}}$. 

By induction, we are able to find a logarithmic model $(\cm^+,\na^+)$ of $M$ such that $\cm^+$ is free and $\mm{Exp}_{\cm^+}\subset\tau$. Theorem \ref{16.05.2020--5} now finishes the proof. 
\end{proof}

\begin{rmk}\label{09.05.2022--2}
As mentioned in Remark \ref{09.05.2022--1}, our definition of exponents can be inadequate in certain contexts. 
Suppose that we set out to obtain a ``normalisation'' result as Theorem \ref{25.05.2020--1} in the following setting. Let $\g o$ be a noetherian  $C$-algebra which is also a domain and define $\mcrs(\lau {\g o} x/\g o )$ along the lines of Definition \ref{08.05.2021--1}-(3).  Given $(M,\na)\in\mcrs(\lau {\g o} x/\g o )$ such that $M$ is a {\it free} $\g o$-module, is it possible to find an analogue of  the combination of Theorem \ref{25.05.2020--1} and Corollary \ref{06.07.2020--1}?

Here we recommend \cite[8.3-4]{andre-baldassarri-cailotto20}.  Their exponents \cite[7.6.1]{andre-baldassarri-cailotto20} are elements in an extension of $\mm{Frac}(\g o)$ following the classical construction (cf. Theorem \ref{03.03.2020--1} and Theorem \ref{02.03.2020--3}).  From there, Andr\'e, Baldassarri and Cailotto go on to show that the exponents of $(M,\na)$ do indeed belong to some integral extension $\g o'$ of $\g o$ and that a ``Jordan decomposition''
 can be  achieved  over  $\lau {\g o'}x$  provided that the differences of exponents are in $C$ \cite[8.4.2]{andre-baldassarri-cailotto20}. This     gives another   approach to Theorem \ref{25.05.2020--1}.
\end{rmk}

We end this section with a capital result, Theorem \ref{29.06.2020--1}, concerning the structure of the  $\lau Rx$--module underlying an object of $\mc(\lau Rx/R)$:

\begin{thm}\label{29.06.2020--1}Let $(M,\na)$ be an object of $\mc(\lau Rx/R)$. Then, $M$ is flat as an $\lau Rx$--module if and only if $M$ is $R$-flat. 
\end{thm}

 Since the ring $\lau Rx$ is not $\g r$--adically complete and since the fibres of $\spc {\lau Rx}\to\spc R$ may fail to be   of finite type over a field,  the argument delivering Theorem \ref{29.06.2020--1}  cannot  be a direct adaptation of known results, e.g. \cite[Lemma 2.4.2, p.40]{katz90}, \cite[p.82]{dos_santos09} or \cite[Proposition 5.1.1]{duong-hai18}. (We profit to note at this point that in the proof of Proposition 5.1.1 in \cite{duong-hai18},  we need to employ  the ``fibre-by-fibre flatness criterion''   \ega{IV}{3}{11.3.10, p. 138} and not the  ``local flatness criterion''.)  We then need the following theorem, which shall also find future applications. 
  
\begin{thm}\label{18.05.2020--1}Let $M\in\mc (\lau Rx/R)$ be given.  Let $\g p$ be a prime ideal of $R$, $S$ the quotient ring $R/\g p$ and $L$ its field of fractions.   Then, the $L\ot_R\lau Rx$--module 
\[
\begin{split}
M|_{\g p}&:=(L\ot_R\lau Rx)\otu{\lau Rx}M
\end{split}
\]  is flat. 
\end{thm}
\begin{proof}
Since $R[x]/\g pR[x]\simeq S[x]$, the Artin-Rees Lemma assures that $\pos Rx/\g p\pos Rx\simeq \pos Sx$ \cite[8.11, p.61]{matsumura89}; inverting $x$, we conclude that  $S\ot_R\lau Rx\stackrel\sim\to \lau Sx$.
As a consequence,  $S\ot_RM$ is an object of $\mc(\lau Sx/S)$. Hence, we only need to show that for any $N\in\mc(\lau Sx/S)$, the $L\ot_S\lau Sx$-module $L\ot_SN$ is flat. 
Using \cite[Theorem 2.5.2.1,p.713]{andre01} (see also Remark \ref{09.09.2020--1}), it is enough to show that $L\ot_S\lau Sx$ has no ideal invariant under $\vt$ other than $(0)$ and $(1)$. Let then $J\subset L\ot_S\lau Sx$ be a non-zero ideal invariant under $\vt$. Since $L\ot_S\lau Sx$ is a localization of $\pos Sx$---note that $\lau Sx$ is a localization of $\pos Sx$ and $L\ot_S\lau Sx$ is a localization of $\lau Sx$---we conclude that $J$ is the extension of  $I:=J\cap  \pos Sx$; clearly $I$ is equally stable under $\vt$. What we are looking for is the a consequence of the 

\emph{Claim.} Let  $I\subset\pos Sx$ be a $\vt$-invariant ideal. Then, there exists an ideal $\g a\subset S$
such that 
\[
\g a\po\lau Sx = I\po \lau Sx.
\]

\noindent{\it Proof.}
 Let $f_1,\ldots,f_n$ be generators of $I$. We conclude that the vector  $\bm f={}^{\top}(f_1,\ldots,f_n)$  satisfies a differential equation \[\vt\bm y= A\bm y,\] where $A\in\mm M_n(\pos Sx)$. Let us now suppose that  $\tau\cap\ZZ=\{0\}$.
There exists  
\[
P\in \mm{GL}_n(\lau Sx)
\] 
such that, if $\bm f=P\bm g$, then 
\begin{equation}\label{14.05.2020--1}
\vt \bm g=B\bm g
\end{equation} 
with $B\in\mm M_n(S)$ a matrix whose image in $\mm M_n(C)$ only has eigenvalues in $\tau$ (Theorem \ref{25.05.2020--1}).
Since $P\in\GL_n(\lau Sx)$, letting $\bm g={}^\top(g_1,\ldots,g_n)$, we have 
\[
\sum_{i=1}^n \lau S xg_i =   I\po\lau Sx.
\]
Let us now write 
\[
\bm g = \sum_{i\ge i_0} \bm g_ix^i.
\]
It then follows from \eqref{14.05.2020--1} that $B\bm g_i=i\bm g_i$ for each $i\ge i_0$. Given $k\in \NN$, let   
\[
B_k : S_k^{\op n}\aro S_k^{\op n}
\]
stand for the $C$-linear endomorphism defined by $B$. Since $\mm{Sp}_{B_k}=\mm{Sp}_{B_0}$ (cf. Proposition \ref{10.07.2019--1}) and $\mm{Sp}_{B_0}\cap\ZZ=\{0\}$, we conclude that, if $i\not=0$, then the image of $\bm g_i$ in $S_k^{\op n}$ vanishes. As $k$ is arbitrary, this implies that  $\bm g_i=0$ for $i\not=0$ and hence $\bm g\in S^n$. The ideal $\g a$ envisaged in the statement is hence obtained. 
\end{proof}
 
\begin{proof}[Proof of Theorem \ref{29.06.2020--1}]One applies the previous result and the fibre-by-fibre flatness criterion \ega{IV}{3}{11.3.10,p.138}. 
\end{proof}

\begin{rmk}\label{09.09.2020--1}We have employed above a Theorem from \cite{andre01} in order to prove Theorem \ref{18.05.2020--1}. Here is a self contained result which gives what we want. 

 Let $A$ be a ring, $\Om$ an $A$--module and   $\mm d:A\to \Om$ a derivation. Given an $A$--module $M$, we define a connection on $M$ as being an additive map $\na:M\to M\ot\Om$ such that $\na(am)=a\na(m)+m\ot\mm da$. Let $A[\Om]=A\op \Om$ and give it the structure of a ring by decreeing that $\om\om'=0$ for $\om,\om'\in\Om$. Let $\iota:A\to A[\Om]$ be the obvious inclusion and $t:A\to A[\Om]$ the map defined by $a\mapsto a+\mm da$; both are morphisms of rings. Using a connection $\na$ on $M$,  we arrive at an isomorphism of $A[\Om]$--modules 
 \begin{equation}\label{09.09.2020--2}
 A[\Om]\otu{t,A} M\arou \sim  M\otu{ A,\iota}A[\Om]
 \end{equation}
 which reduces to the identity modulo  $\Om$ \cite[Proposition 2.9]{berthelot-ogus78}. 

Let us suppose that $M$ is of finite type and let ${\rm Fitt}_r$ be the $r$th Fitting ideal of $M$ \cite[20.4]{eisenbud95}. By a fundamental property of these ideals \cite[Corollary 20.5]{eisenbud95}, the isomorphism in \eqref{09.09.2020--2} says that $t({\rm Fitt}_r)A[\Om]=\iota({\rm Fitt}_r)A[\Om]$. This implies the inclusion    
\[
\mm d({\rm Fitt}_r)\subset {\rm Fitt}_r\po\Om.
\] 
 
We say that an ideal $I\subset A$ is ${\rm d}$--invariant if ${\rm d}(I)\subset I\Om$. Therefore, imposing that   the only ${\rm d}$--invariant ideals of $A$ are $(0)$ and $(1)$ and employing \cite[Proposition 20.8]{eisenbud95}, we conclude that  either $M=0$, or $M$ is projective of constant rank. (Note that if $\spc A$ is disconnected, then there are immediately ${\rm d}$--invariant ideals other than $(0)$ and $(1)$, so  constancy of the rank is appropriate).
\end{rmk}

\section{Logarithmic models for connections from $\mcrs(\lau Rx/R)$}\label{28.06.2021--7}

Let $k\in\NN$. For each $(\cm,\na)\in \mclog(\pos Rx/R)$, the  arrow  
\[
\na:\cm/\g r^{k+1}\aro \cm/\g r^{k+1}
\]
gives rise to an object of $\mclog(\pos Cx/C)_{(R_k)}$ and this construction produces a functor 
\[\bullet|_{k}:\mclog(\pos Rx/R)\aro\mclog(\pos Cx/C)_{(R_k)}.
\]
Analogously, we obtain a functor 
\[
\bullet|_{k}:\mc(\lau Rx/R)\aro\mc(\lau Cx/C)_{(R_k)}
\]
and these two fit into a commutative diagram  (up to natural isomorphism) 
\[
\xymatrix{
\mclog(\pos Rx/R)\ar[rr]^\ga\ar[d]_{\bullet|_k} && \mc(\lau Rx/R)\ar[d]^{\bullet|_{k}}
\\
\mclog(\pos Cx/C)_{(R_k)}\ar[rr]_\ga && \mc(\lau Cx/C)_{(R_k)}.
}
\]
In particular,   if $M\in\mc(\lau Rx/R)$ is regular-singular, then 
$M|_{k}$ is also regular-singular.

\begin{thm}[Deligne-Manin models] \label{16.05.2020--1}
Any $M\in\mcrs(\lau Rx/R)$ possesses a logarithmic model $\cm$ such that, for every $k\in\NN$, the object 
\[
\cm|_{k}\in\mclog(\pos Cx/C)_{(R_k)},
\] 
enjoys the ensuing properties: 
\begin{enumerate}[(1)]
\item All its exponents lie in $\tau$. 
\item It is free in relation to $R_k$.
\item The isomorphism $\ga(\cm|_k)\simeq M|_k$ is compatible with the action of $R_k$.  
\end{enumerate} 
Put otherwise, $\cm|_k$ is a Deligne-Manin model in the sense of Theorem \ref{10.03.2020--2}. 
\end{thm}
\begin{proof} 
Let us begin with a piece of commutative algebra which is fundamental to our argument: the ring $\pos Rx$ is $\g r$-adically complete  \cite[Exercises 8.6 and 8.2]{matsumura89}. This allows us to construct $\pos Rx$-modules by taking limits.

{\it Step 1: Putting Deligne-Manin models of truncations together.}
For each $k$, let 
\[\begin{array}{c}\text{$\cm_k$  be a Deligne-Manin logarithmic}\\\text{ model of $
M|_{k}\in\mcrs(\lau Cx/C)_{(R_k)}$,}\end{array}
\] 
as obtained in Theorem \ref{10.03.2020--2}. By definition, the exponents of $\cm_k$ are all on $\tau$. 
Note that $\cm_{k+1}|_k$, regarded as an object of $\mclog(\pos Cx/C)_{(R_k)}$, is a logarithmic lattice for $M|_k$ enjoying all the properties described in Theorem \ref{10.03.2020--2}. (To see that the exponents remain unchanged, see Corollary \ref{17.05.2020--1}.) 
We can therefore, by  Theorem \ref{10.03.2020--2}, find an isomorphism 
\[
\ph_k:\cm_{k+1}|_{k}\arou\sim\cm_k,
\] 
in the category $\mclog(\pos Cx/C)_{(R_k)}$, such that 
\[
\xymatrix{\cm_{k+1}|_k\ar[rr]^{\ph_k}\ar[d]_{\rm can.}&&\cm_k\ar[d]^{\rm can.}
\\
(M|_{k+1})|_k\ar[rr]_{\text{can.}} &&M|_k
}
\] 
commutes.
Because of \ega{0}{I}{Proposition 7.2.9}, \[
\cm:=\lip_k\cm_k
\] 
is a finite $\pos Rx$-module since, as mentioned before,  $\pos Rx\simeq \lip_k\pos{R_k}x$. Furthermore, for each $k$, the natural arrow $\cm/\g r^{k+1}\to\cm_k$ is an isomorphism by loc.cit.
Using the derivations on the various $\cm_k$, we construct a derivation  $\nabla$ on $\cm$: we have therefore produced an  element of $\mclog(\pos Rx/R)$. Clearly for any given $k\in\NN$, the object $\cm|_k\in\mclog(\pos Cx/C)_{(R_k)}$ enjoys properties (1), (2) and (3) of the statement.

{\it Step 2: Showing that the previously constructed logarithmic connection is a model.}  This is  not automatic  since all we know for the moment is the existence of a compatible family of isomorphisms    
\[ \cm[x^{-1}]/\g r^{k+1}\arou\sim M/\g r^{k+1}. 
\]
These {\it do not} necessarily give us an isomorphism of $\lau Rx$-modules $\cm[x^{-1}]\simeq M$. 

For that, let $\mathds  M$ be an $x$-pure logarithmic model for $M$ (cf. Lemma \ref{15.07.2020--1}).    Then   $\mathds M|_k$ is a logarithmic model for $M|_k$ (but we do not have much more to say about it). According to  Corollary \ref{17.05.2020--1}(2) and  Proposition \ref{09.04.2019--1}-(2), there exists an integer $\de\ge0$ such that the dotted arrow in 
\[
\xymatrix{
\MM|_{k}\ar[d]_{\text{can.}}\ar@{-->}[rr]^{\ps_k} &&  \cm(\de)|_k \ar@{^{(}->}[d]^{\text{can.}} 
\\
M|_k\ar@{=}[rr]  && M|_k
}
\] 
can be found for each $k$. (The definition of $\cm(\de)$ is given in Example \ref{24.09.2020--1}.) Note that $\ps_k$ is automatically an arrow of $\mclog(\pos Cx/C)_{(R_k)}$. 

As $\pos Rx$ is $\g r$-adically complete, we then derive an arrow, now in $\mclog(\pos Rx/R)$,
\[
\ps:\MM\aro  \cm(\de)
\]
inducing $\ps_k$ for each $k$.  We contend that $\ps[x^{-1}]:\MM[x^{-1}]\to\cm(\de)[x^{-1}]$ is an isomorphism. Since $\ps[x^{-1}]/\g r^{k+1}$ is an isomorphism for each $k$, we conclude that the $\g r\lau Rx$-adic completion of $\ps[x^{-1}]$ is an isomorphism. Hence,
 $\ps[x^{-1}]$ is an isomorphism {\it on a neighbourhood of   the closed fibre of $\spc \lau Rx\to \spc R$} \ega{0}{I}{7.3.7}.
This implies that  the kernel and cokernel of $\ps[x^{-1}]$, which are objects of $\mc(\lau Rx/R)$, vanish on an open neighbourhood of the closed fibre of $\spc \lau Rx\to \spc R$. 
Using Theorem \ref{18.05.2020--1} and then Lemma \ref{11.05.2020--1} below, we can infer that the kernel and cokernel of $\ps[x^{-1}]$  are trivial and  $\ps[x^{-1}]$ is an isomorphism and $\cm(\de)$ is a model for $M$. 
\end{proof}  

The following result was employed in verifying Theorem \ref{16.05.2020--1} and shall also be useful in establishing Theorem \ref{02.03.2020--2} to come. 

\begin{lem}\label{11.05.2020--1}Let
$R\to \co$ be a faithfully flat morphism of noetherian rings whose fibre rings are domains. Let 
$M$ be an $\co$--module of finite type such that for each $\g p\in \spc R$, the fibre $M\ot_\co(\co\ot_{R}\boldsymbol k(\g p))$ is a flat $\co\ot_R\boldsymbol k(\g p)$-module.   Assume that $M_{\g P_0}=0$ for one prime $\g P_0\in\spc \co$ above $\g r$. Then $M=0$. 
\end{lem}
\begin{proof}Let $U=\{\g P\in \spc \co\,:\,M_{\g P}=0\}$ be the complement of the support of $M$; it is an open and non-empty subset of $\spc \co$.  Let $\g P\in  U$ and write $\g p$ for its image in $\spc R$. 
Now, if $\g Q\in\spc \co$ is also above $\g p$, we can say that $M_{\g Q}=0$. Indeed,  $M\ot_\co\boldsymbol k(\g P)=0$ and hence the projective  $\co\ot_R\boldsymbol k(\g p)$--module $M\ot_R\boldsymbol k(\g p)$ vanishes. Then, $M\ot_R\boldsymbol k(\g Q)$ vanishes as well and $M_{\g Q}=0$. Now we note that the image of $U$ in $\spc R$ is open \cite[6.G, Theorem 7, pp 46--7]{matsumura80} and contains the closed point $\g r$, which means that the image of $U$ is $\spc R$. We conclude that $U=\spc \co$.  
\end{proof}

Let us dig further on the method of proof of Theorem \ref{16.05.2020--1}. In it, we dealt with an object $(M,\na)\in\mcrs(\lau Rx/R)$ and, for each $k\in\NN$, a logarithmic models $\cm_k$ of $(M,\na)|_{k}$ to conclude that the $\cm_k$ could be used to construct a logarithmic model of $M$. We now show that the hypothesis that $(M,\na)$ is regular-singular is necessary.

\begin{counterex}\label{28.06.2021--8}Let $R=\pos Ct$,  $M=\lau Rx\po\bb m$ and define $\na(\bb m)=( t/x)\po \bb m$; this gives us an object $(M,\na)\in\mc(\lau Rx/R)$. (It is not difficult to prove that $(M,\na)$ is not regular-singular.)
Let $\cm_k=(\pos {R_k}x,\vt)\in\mclog(\pos Cx/C)_{(R_k)}$. Let 
\[
e_k:=\sum_{j=0}^k\fr{t^j x^{-j}}{j!}\in\lau Rx.
\]
Then, in $(M,\na)|_{k}$, the element $e_k\bb m$ satisfies $\na(e_k\bb m)=0$. Hence,  $\cm_k:=(\pos {R_k}x,\vt)$ is logarithmic model for $(M,\na)|_{k}$, but $(\pos Rx,\vt)$ is not a logarithmic model for $(M,\na)$.
\end{counterex}

{ 
In passing, we observe that the Deligne-Manin models in Theorem \ref{16.05.2020--1} have a remarkable property if the regular-singular connection underlies a flat $R$--module.  

\begin{cor}Let  $(M,\na)\in\mcrs(\lau Rx/R)$ be given. Then, if $M$ is $R$-flat, it is the case that the logarithmic  model $\cm$ from Theorem \ref{16.05.2020--1} is \textbf{free} as an $\pos Rx$-module. 
\end{cor}
\begin{proof}Let $k$ be fixed. We shall show that $\cm|_k$ is flat over $\pos {R_k}x$ and then apply the local flatness criterion \cite[Theorem 22.3, p.174]{matsumura89} to assure flatness of $\cm\simeq\lip\cm|_k$; this in turn shows that $\cm$ is free since $\pos Rx$ is local.
We note that, since $M$ is $R$-flat, it is also $\lau Rx$--flat (Theorem \ref{29.06.2020--1}) and therefore $M|_k$ is also $\lau {R_k}x$--flat. 

By assumption, we can  write $\cm|_k\simeq \pos {R_k}x\ot_{R_k}V_k$ for a certain $R_k$-module $V_k$. Then, $\lau {R_k}x\ot_{R_k}V_k\simeq M|_k$ is $\lau{R_k}x$--flat. Because $R_k\to \lau {R_k}x$ is faithfully flat (flatness follows from flatness of $R_k\to \pos {R_k}x$) we conclude that $V_k$ is $R_k$--flat \cite[4.E(i), p. 29]{matsumura80}. Hence, $\cm|_k$ is flat.  
\end{proof}
} 
 
In possession of Theorem \ref{16.05.2020--1}, we are now able to interpret the category $\mcrs(\lau Rx/R)$  as a category of representations echoing Corollary \ref{22.04.2020--2}. We need a definition. 

\begin{dfn}\label{25.09.2020--2}We let $\mcrs(\lau Rx/R)^\wedge$ stand for the category whose \item [{\it objects}] are families $\{(M_k,\ph_k)\}_{k\in\NN}$, where 
$M_k\in\mcrs(\lau Cx/C)_{(R_k)}$ and $\ph_k:M_{k+1}|_{k}\to M_k$ are isomorphisms in $\mcrs(\lau Cx/C)_{(R_k)}$;
\item[{\it arrows}] between $\{(M_k,\ph_k)\}_{k\in\NN}$ and $\{(N_k,\ps_k)\}_{k\in\NN}$ are  compatible  sequences  \[\{\al_k:M_k\to N_k\}\in\prod_k\hh{\mc_{(R_k)}}{M_k}{N_k}.\] 
\end{dfn}

\begin{thm}\label{02.03.2020--2}The natural functor 
\[
\bb{MC}_{\rm rs} (\lau Rx/R)\aro  \mcrs(\lau Rx/R)^\wedge,
\]
\[
(M,\na)\longmapsto\{(M,\na)|_k\}_k
\]
is an equivalence. 
\end{thm}
\begin{proof}We start by showing {\it essential surjectivity}. To ease notation, we omit reference to the derivations. 
Let  
\[
\{ M_k, \ph_k\}_{k\in\NN}\in \mcrs(\lau Rx/R)^\wedge.
\] 
Let  $\cm_k$ be the logarithmic {\it lattice} constructed from $M_k$ as in Theorem \ref{10.03.2020--2}. Note that $
\cm_{k+1}|_{k} \in\mclog(\pos Cx/C)_{(R_k)}$
is a logarithmic lattice for $M_{k+1}|_{k}$ which satisfies all conditions of Theorem \ref{10.03.2020--2}.
Hence, 
\[
\ph_k : M_{k+1}|_{k}\arou\sim M_k
\]
can be extended to an {\it isomorphism} 
\[
\Ph_k:\cm_{k+1}|_{k}\arou\sim\cm_k
\]
in $\mclog(\pos Cx/C)_{(R_k)}$. 

Define \[\cm=\lip_k\cm_k.\]  As an $\pos Rx=\lip_k\pos{R_k}x$-module, it is of finite type   and the projection $\cm\to\cm_k$ has kernel $\g r^{k+1}\cm$ \ega{0}{I}{Proposition 7.2.9}. Therefore, $\cm$ gives rise to  an object of $\mclog(\pos Rx/R)$. Let $M=\ga(\cm)$. Then $M$ is an object of $\mcrs(\lau Rx/R)$  whose image in $ \mcrs(\lau Rx/R)^\wedge$ is $\{M_k,\ph_k\}$. 

We now prove {\it fullness}. 
Let $M$ and $N$ be objects of $\mcrs(\lau Rx/R)$ and pick Deligne-Manin models $\cm$ and $\cn$ of $M$ and $N$ as in Theorem \ref{16.05.2020--1}.  
For each  $k$, let 
\[
\ph_k:M|_k\aro N|_k
\] be an arrow in $\mcrs(\lau Cx/C)_{(R_k)}$ and suppose that $\ph_{k+1}|_k=\ph_k$. Because of Theorem \ref{10.03.2020--2}, there exists an arrow in $\mclog(\pos Cx/C)_{(R_k)}$,  $\wt\ph_k:\cm|_k\to\cn|_k$, extending $\ph_k$. In addition, uniqueness of the extension  forces $\wt\ph_{k+1}|_{k}$ to coincide with $\wt\ph_k$ after all the necessary identifications. Hence, there exists $\wt\ph:\cm\to \cn$ such that $\wt\ph|_k=\wt\ph_k$, which establishes the existence of $\ph:M\to N$ inducing each $\ph_k$. 

Finally, we establish {\it faithfulness}. Let then $\ph:M\to N$ be such that $\ph_k:M|_k\to N|_k$ is null; we conclude that $I=\mm{Im}(\ph)\subset\cap_k\g r^k\cn$. By Nakayama's Lemma \cite[Theorem 2.2, p.8]{matsumura89}, there exists $a\equiv1\mod\g r$  such that  $aI=0$. Hence, $I_{\g p}=0$ if $\g p\in\spc \lau Rx$ is above $\g r$. Now,   $I\in\mc(\lau Rx/R)$ and hence Theorem \ref{18.05.2020--1} followed by  Lemma \ref{11.05.2020--1} prove that $I=0$.
\end{proof}

Let now 
\[\Ph_\al:\rep C\ZZ\aro \mcrs(\lau Cx/C)
\] be a tensor equivalence as in Corollary \ref{22.04.2020--2}; it produces obvious equivalences 
\[
\Ph_\al: \rep C\ZZ_{(R_k)}\arou\sim \mcrs(\lau Cx/C)_{(R_k)}
\] 
of $R_k$-linear categories. 
Following the pattern established in Definition \ref{25.09.2020--2}, we introduce the category $\rep R\ZZ^\wedge$. With little effort it can be proved that $\rep R\ZZ^\wedge$ is equivalent to $\rep R\ZZ$.
We hence arrive at: 
\begin{cor}\label{16.07.2020--1}The composition 
\[\mcrs(\lau Rx/R)\aro \mcrs(\lau R x/R)^\wedge\aro  \rep R\ZZ^\wedge\simeq \rep R\ZZ\]
is an equivalence of $R$-linear tensor categories. \qed
\end{cor}

\section{Connections on $\PP_R\smallsetminus\{0,\infty\}$}\label{29.06.2021--1}
In what follows, $\PP_R$ stands for the projective line over $R$; it is covered by the two affine open subsets $\af_0=\spc R[x]$ and $\af_\infty =\spc R[y]$, and $x=y^{-1}$ on $\af_0\cap\af_\infty=\PP_R\smallsetminus\{0,\infty\}$. 

 Following the pattern of   Definition \ref{01.07.2021--1}, we introduce  the  category of {\it  connections on $\PP_R\smallsetminus\{0,\infty\}$, or on $R[x^\pm]$}, of {\it logarithmic connections on $\PP_R$} and    of {\it regular-singular connections}; we denote them respectively by    
\[\mc(R[x^\pm]/R),\quad
\mclog(\PP_R/R)\quad\text{and}\quad\mcrs(R[x^\pm]/R).
\] 
Letting 
\[
 \mclog(\PP_R/R)\arou{\bb r_0}\mclog(\pos Rx/R)\quad\text{and}\quad \mclog(\PP_R/R)\arou{\bb r_\infty}\mclog(\pos Ry/R)
\]
stand for the obvious functors, we define the  {\it exponents} of $(\cm,\na)\in\mclog(\PP_R/R)$ as the set of exponents of either $\bb r_0\cm$ or $\bb r_\infty\cm$ (cf. Definition \ref{28.06.2021--6}).

Denote by
\[
\ga_\PP: \mclog(\PP_R/R)\aro\mc(R[x^\pm]/R)
\]
the   functor which associates to $(\ce,\na)$ its restriction to $\PP_R\smallsetminus\{0,\infty\}$. 
Given $M\in\mcrs(R[x^\pm]/R)$, any $\cm\in\mclog(\PP_R/R)$ such that $\ga_\PP(\cm)\simeq M$ is called a {\it logarithmic model} of $M$. 

Note that if $M\in\mc(R[x^\pm]/R)$ is regular-singular, then 
\[
M|_{k}=M/\g r^{k+1}M\in\mc(C[x^\pm]/C)_{(R_k)}
\] 
is also regular-singular for any given $k\in\NN$.

We now complete the picture drawn in Section \ref{28.06.2021--7} 
 by analyzing regular-singular connections on $R[x^\pm]$. 
We aim at 

\begin{thm}\label{20.10.2020--6}The restriction 
\[
\bb r_0: \mcrs(R[x^\pm]/R)\aro \mcrs(\lau Rx/R)
\]
is an equivalence.  
\end{thm}

Its proof will follow with little effort from   Theorem \ref{20.10.2020--8} below. 
This, in turn,  requires the category 
\[
\mcrs(R[x^\pm]/R)^\wedge,
\] 
whose definition parallels   Definition \ref{25.09.2020--2} (the details are left to the reader). Let 
\begin{equation}\label{20.10.2020--7}
\bb r_0^\wedge :   \mcrs(R[x^\pm]/R)^\wedge\aro \mcrs(\lau Rx/R)^\wedge
\end{equation}
be the obvious functor. Because of Corollary \ref{07.07.2020--1}, we know that $\bb r_0^\wedge$ is an equivalence.  
\begin{thm}\label{20.10.2020--8}The natural functor 
\[
\mcrs(R[x^\pm]/R)\aro \mcrs(R[x^\pm]/R)^\wedge
\]
\[
(M,\na)\longmapsto\{(M,\na)|_k\}_k
\]
is an equivalence. 
\end{thm} 

Assuming the veracity of this result, we can give a 

\begin{proof}[Proof of Theorem \ref{20.10.2020--6}] Follows from the the commutative diagram of categories
\[
\xymatrix{
\mcrs(R[x^\pm]/R)\ar[rr]^{\bb r_0} \ar[d]_{\text{Theorem \ref{20.10.2020--8}}}^\sim & &   \mcrs(\lau Rx/R)\ar[d]^{\text{Theorem \ref{02.03.2020--2}}}_\sim
\\
\mcrs(R[x^\pm]/R)^\wedge\ar[rr]_{\bb r_0^\wedge} && \mcrs(\lau Rx/R)^\wedge 
}
\]
and the fact that $\bb r_0^\wedge$ is an equivalence. 
\end{proof}

Let us now start the verification of Theorem \ref{20.10.2020--8}. Simple facts come first.  

\begin{lem}\label{06.07.2021--1}Any $M\in\mcrs(R[x^\pm]/R)$ allows a logarithmic model $\cm$ such that $\cm(\af_0)$ has no $x$-torsion  and $\cm(\af_\infty)$  no $y$-torsion. 
\end{lem}
\begin{proof}Let $\cn$ be any logarithmic  model. The sub-module of $x$-torsion in $\cn(\af_0)$ is invariant under $\vt$. The submodule of $y$-torsion in $\cn(\af_\infty)$ is invariant under $\vt$.  We can therefore take the quotients to produce the required model. 
\end{proof}

\begin{lem}\label{05.07.2021--1}
Let $E\in\mc(R[x^\pm]/R)$ be given. Then, for each $\g p\in\spc R$, the $\bm k(\g p)[x^\pm]$-module $\bm k(\g p)[x^\pm]\otu{R[x^\pm]} E$ is locally free. 
\end{lem}
\begin{proof}See either \cite[Proposition 8.9]{katz70},  or Remark \ref{09.09.2020--1}. 
\end{proof}

We are unfortunately unable to find a proof of Theorem \ref{20.10.2020--8} based simply on Corollary \ref{07.07.2020--1} and the equivalence $\bb r_0^\wedge$. 
Hence, we shall need to go through the arguments used to establish Theorem \ref{02.03.2020--2} (the analogue of Theorem \ref{20.10.2020--8} in the formal case) and adapt them.  Luckily, there are no major modifications, except that the process of  taking the limit allowed by $\g r$-adic completeness of $\pos Rx$  needs to be replaced by Grothendieck's GFGA Theorem for sheaves on $\PP_R$. See \cite{illusie05} for a complete proof of this result  and \cite[3.2]{harbater03} for a valuable outline. Note that this is also the technique employed in \cite{hai-dos_santos20}, which renders the matter technically more demanding. 

When   employing GFGA in this context, we are     hindered by the following difficulty. Say that $\cm$ is  a coherent $\co_{\PP_R}$-module   such that, for every $k\in\NN$, the $\co_{\PP}$-module (with action of $R_k$) $\cm|_k:=\cm/\g r^{k+1}$  carries a logarithmic connection $\na_k:\cm|_k\to\cm|_k$, and that, in addition, the natural isomorphisms 
\[
\cm|_{k+1} \arou\sim\cm|_k
\]
are compatible with the logarithmic connections. Since $\cm$  {\it is not} the sheaf $\lip_k\cm_k$, we need to ask:   is it possible to endow $\cm$ with a logarithmic connection $\na:\cm\to\cm$ inducing the various $\na_k$? The answer is yes, as we now explain. 

Let $\ce$ be a coherent $\co_{\PP_R}$-module and introduce 
 $J\ce$ as being   the sheaf of $R$-modules   $\ce\op\ce$.  Endow it with the structure of an $\co_{\PP_R}$-module by 
\[
a\po(e,e') = (ae,ae'+\vt(a)e).
\]
Write $p:J\ce\to \ce$ for the projection onto the first factor. 
It is not hard to see that $J\ce$ remains  coherent
and that a logarithmic connection is none other than an $\co_{\PP_R}$-linear arrow 
\[\si: \ce\aro J\ce\]
such that $p\si=\id$. Indeed, if $p\si=\id$, then $\si=(\id,\na)$, where $\na$ is a logarithmic connection.
We now return to the question raised above and state it as a Lemma for future referencing. 

\begin{lem}\label{04.07.2021--1}Let $\cm$ be  a coherent $\co_{\PP_R}$-module   such that, for every $k\in\NN$, the $\co_{\PP}$-module (with action of $R_k$) $\cm|_k:=\cm/\g r^{k+1}$  carries a logarithmic connection $\na_k:\cm|_k\to\cm|_k$, and that, in addition, the natural isomorphisms 
\[
\cm|_{k+1} \arou\sim\cm|_k
\]
are compatible with these  connections. Then $\cm$ carries a logarithmic $R$-linear connection $\na$ inducing $\na_k$ for each $k$. In addition, if $\cn$ is an object of $\mclog(\PP_R/R)$ and $\Ph:\cm\to\cn$ is an arrow of coherent $\co_{\PP_R}$-modules such that $\Ph|_k:\cm|_k\to\cn|_k$ lies in $\mclog(\PP/C)_{(R_k)}$ for each $k$, then $\Ph$ is actually an arrow of $\mclog(\PP_R/R)$.  
\end{lem}
\begin{proof}
Let $\si_k:\cm|_k\to J\cm|_k$ be defined by $\si_k=(\id,\na_k)$. We then obtain, by GFGA, an arrow $\si:\cm\to J\cm$ such that $p\si=\id$, that is, a logarithmic connection. 
The final claim is also proved with similar techniques. 
\end{proof}
We can now give the first step towards  Theorem \ref{20.10.2020--8}.

\begin{thm}[Deligne-Manin models]\label{12.10.2020--3}
	Let $M\in\mcrs(R[x^\pm]/R)$. There exists a unique logarithmic model $\cm$ of $M$ such that, for every $k\in\NN$, the object 
	\[
	\cm|_{k}\in\mclog(\PP/C)_{(R_k)},
	\] 
	enjoys the ensuing properties: 
	\begin{enumerate}[(1)]
		\item All its exponents lie in $\tau$. 
		\item It is free in relation to $R_k$.
		\item The isomorphism $\ga_\PP(\cm|_k)\simeq M|_k$ is compatible with the action of $R_k$.  
	\end{enumerate} 
	Put otherwise, $\cm|_k$ is a Deligne-Manin model in the sense of Theorem \ref{10.07.2020--3}. 
\end{thm}
\begin{proof}
	This is much the same as the proof of  Theorem \ref{16.05.2020--1} and we shall give only some indications of how to replace the arguments in its proof to the present context.

	{\it For Step 1.} The use of Theorem \ref{10.03.2020--2} is replaced by that of Theorem \ref{10.07.2020--3} and Theorem \ref{12.10.2020--2}. The use the $\g r$-adic completeness of $\pos Rx$  is replaced by  GFGA {\it supplemented} by Lemma \ref{04.07.2021--1}. We then arrive at an object $\cm\in\mclog(\PP_R/R)$. 
	
	{\it For Step 2.} We replace Lemma \ref{15.07.2020--1} by Lemma \ref{06.07.2021--1} in finding a convenient logarithmic model $\MM$  for $M$. We then replace Proposition \ref{09.04.2019--1} and Corollary \ref{17.05.2020--1} by Proposition \ref{12.10.2020--1}. To continue, we employ GFGA and Lemma \ref{04.07.2021--1} instead of completeness of $\pos Rx$ and  Lemma \ref{05.07.2021--1} instead of Theorem \ref{18.05.2020--1}. 
\end{proof}

\begin{proof}[Proof of Theorem \ref{20.10.2020--8}]  {\it Essential surjectivity.} Let $\{M_k,\ph_k\}_{k\in \NN}$
be	in $\mcrs(R[x^\pm]/R)^\wedge$. For each $k$, let $\cm_k\in \mclog(\PP/C)_{(R_k)}$ be a Deligne-Manin lattice for $M_k$ (cf. Theorem \ref{10.07.2020--3}). 
Because of Theorem \ref{12.10.2020--2}, the isomorphisms 
	\[
	\ph_k:M_{k+1}|_k\arou\sim M_k
	\]
may be extended to isomorphisms 
	\[
	\Ph_k:\cm_{k+1}|_k\arou\sim\cm_k
	\]
	of $\mclog(\PP_C/C)_{(R_k)}$. 
	
By  GFGA, there exists a coherent sheaf $\cm$ on $\PP_R$ and isomorphisms    $\cm|_k\simeq \cm_k$  such    that the natural transition isomorphisms  correspond to the $\Ph_k$ above. 
Lemma \ref{04.07.2021--1}  now shows that $\cm$ comes with a logarithmic connection and we arrive at an object of $\mclog(\PP_R/R)$. Then, $M=\ga_{\PP}(\cm)$ is an object in $\mcrs(R[x^\pm]/R)$ satisfying $M|_k\simeq M_k$ for each $k\in \NN$.

 {\it Fullness.} Let $M$ and $N$ be objects of $\mcrs(R[x^\pm]/R)$.  
 For each $k\in \NN$, let
\[
\ph_k:M|_k\longrightarrow N|_k
\]
be an arrow in $\mcrs(C[x^\pm]/C)_{(R_k)}$ and suppose that $\ph_{k+1}|_{k}=\ph_{k}$.
 Pick Deligne-Manin models $\cm$ and $\cn$ of $M$ and $N$ as in Theorem \ref{12.10.2020--3}.  By   Theorem \ref{12.10.2020--2}, there exists, for any given $k$, an arrow $\Ph_{k}: \cm|_k \to \cn|_k$ in $\mclog(\PP/C)_{(R_k)}$ such that $\ga_\PP(\Ph_k)=\ph_{k}$. In addition, uniqueness of the extension forces  $ \Ph_{k+1}|_{k} =\Ph_{k}$ for each $k$. By GFGA, there exists an arrow $\Ph : \mathcal{M} \to \mathcal{N}$ of coherent $\co_{\PP_R}$-modules
satisfying $\Ph |_{k}=\Ph_{k}$ for each $k\in \mathds{N}$. From Lemma \ref{04.07.2021--1}, we can also affirm that $\Ph$  is an arrow of 
 $\mclog(\PP_R/R)$.  The arrow $\ph=\ga_\PP(\Ph)$ lies in $\mcrs(R[x^\pm]/R )$ and induces $\ph_k$ for each $k\in \mathds{N}$.
	
 {\it Faithfulness.} Let $\ph: M\to N$  be an arrow in $\mcrs(R[x^\pm]/R)$ such that $\ph_k: M|_{k}\to N|_{k}$ is null for all $k\in \NN$. 
We conclude that $I=\mm{Im}(\ph)\subset\cap_k\g r^kN$. By Nakayama's Lemma \cite[Theorem 2.2, p.8]{matsumura89}, there exists $a\equiv1\mod\g r$  such that  $aI=0$. Hence, $I_{\g P}=0$ if $\g P\in\spc   R[x^\pm]$ is above $\g r$. To show that $I=0$, we only require  Lemma \ref{11.05.2020--1} and Lemma \ref{05.07.2021--1}. 
\end{proof}



\begin{thebibliography}{99}
	
	
\bibitem[Ab80]{abe80} E. Abe, {\it Hopf algebras}. Cambridge University Press, 1980. 
	
\bibitem[ABC20]{andre-baldassarri-cailotto20}Y. Andr\'e, F. Baldassarri and M. Cailotto, {\it de Rham cohomology of differential modules on algebraic varieties}. 
Second edition. Progress in Mathematics, 189. Birkh\"auser/Springer. 2020


\bibitem[And01]{andre01} Y. Andr\'e, \textit{Diff\'erentielles non commutatives et Th\'eorie de Galois diff\'erentielle ou aux diff\'erences.} Ann. Scient. Ec. Norm. Sup., 4 serie, t. 34, (2001), p. 685-739.

\bibitem[BLM02]{bass-lubotzky-magid02}H. Bass, A. Lubotzky and A. Magid, {\it The proalgebraic completion of rigid groups}. Geom. Dedicata 95 (2002), 19--58.


\bibitem[BO78]{berthelot-ogus78} P. Berthelot and A. Ogus, {\it Notes on crystalline cohomology}, Princeton University Press, Princeton, N.J.; University of Tokyo Press, Tokyo, 1978.


\bibitem[BouA]{bourbakialgebre}N. Bourbaki, {\it \'El\'ements de Math\'ematique. Alg\`ebre}. 
Ch. 1 to 3. Second edition. Reprint of the original 1970 edition, Springer, 2007. 
Ch. 4 to 7, Reprint of the original 1981 edition. Springer, 2007. 
 Ch. 8. 
 Ch. 


\bibitem[BouS]{bourbakiset}N. Bourbaki, {\it Elements of Mathematics.  Theory of sets}. Springer, 2004

\bibitem[BouLie]{bourbakilie} N. Bourbaki, {\it \'El\'ements de Math\'ematiques.   Groupes et alg\`ebres de Lie.}  
Berlin,  Springer. Chapitre 1   (2007), Chapitre 7 et 8 (2006). 


\bibitem[CL55]{coddington-levinson55}E. A. Coddington and N. Levinson, {\it Theory of ordinary differential equations}. McGraw-Hill Book Company, Inc., New York-Toronto-London, 1955.


\bibitem[Del90]{deligne90}P. Deligne, {\it Cat\'egories tannakiennes}. The Grothendieck Festschrift, Vol. II,   111--195. Birkh\"auser, 1990. 


\bibitem[Del87]{deligne87}P. Deligne, {\it Le groupe fondamental de la droite projective moins trois points}.  Galois Groups over $\mathbb Q$, Berkeley, CA, 1987.  Math. Sci. Res. Inst. Publ., vol. 16, Springer, New York, 1989, pp. 79--297.

\bibitem[DM82]{deligne-milne82} P. Deligne and J. Milne, Tannakian categories, Lecture Notes in Mathematics 900, pp. 101 -- 228, Springer-Verlag, Berlin-New York, 1982.

\bibitem[Del70]{deligne70} P. Deligne, \'Equations diff\'erentielles \`a points  singuliers r\'eguliers, Springer Lecture Notes in Mathematics 163, 1970.

\bibitem[DG70]{demazure-gabriel70}M. Demazure and P. Gabriel, \emph{Groupes alg\'ebriques. Tome I}. Masson \& Cie, \'Editeur, Paris; North-Holland Publishing Co., Amsterdam, 1970.  

\bibitem[DH18]{duong-hai18}N. D. Duong and P. H. Hai, {\it Tannakian duality over Dedekind rings and applications}. Math. Z. (2018) 288, 1103--1142. 

\bibitem[Ei95]{eisenbud95}
D. Eisenbud, {\it Commutative algebra with a view toward algebraic geometry}, Graduate Texts in Mathematics 150, Springer, New York, 1995.


\bibitem[FMFS22]{fmfs2}L. Fiorot, 
T. Monteiro Fernandes and C. Sabbah, {\it Relative regular Riemann-Hilbert correspondence II.} Preprint 2022. 

\bibitem[FMFS21]{fmfs}L. Fiorot, 
T. Monteiro Fernandes and C. Sabbah, {\it Relative regular Riemann-Hilbert correspondence.} Proc. Lond. Math. Soc. (3) 122 (2021), no. 3, 434--457.


\bibitem[Gu81]{guralnick81}R. M. Guralnick, {\it Similarity of matrices over local rings},  Linear Algebra Appl. 41 (1981), 161--174.

\bibitem[HdS22]{hai-dos_santos20} P. H. Hai and J. P. dos Santos, {\it Regular singular connections on relative complex schemes}. Preprint February 2022. To appear in Ann. Sc. Norm. Super. Pisa Cl. Sci. (5), 31 pages. 


\bibitem[Ha03]{harbater03}D. Harbater,
{\it Patching and Galois theory}.  Galois groups and fundamental groups, 313--424, 
Math. Sci. Res. Inst. Publ., 41, Cambridge Univ. Press, Cambridge, 2003. 

\bibitem[HM69]{hochschild-mostow69}G. Hochschild and D. Mostow, {\it Pro-affine algebraic groups}, Amer. J. Math. 91
(1969),1141--1151.

\bibitem[Il05]{illusie05}L. Illusie, {\it Grothendieck's existence theorem in formal geometry}. With a letter (in French) of Jean-Pierre Serre. Math. Surveys Monogr., 123, Fundamental algebraic geometry, 179--233, Amer. Math. Soc., Providence, RI, 2005

\bibitem[IY08]{ilyashenko-yakovenko08}
Yu. Ilyashenko and S.  Yakovenko, {\it Lectures on analytic differential equations}. Graduate Studies in Mathematics, 86. American Mathematical Society, Providence, RI, 2008. 

\bibitem[J87]{jantzen87} J. C. Jantzen, {Representations of algebraic groups}, Pure and Applied Mathematics, 131. 
Academic Press, Inc., Boston, MA, 1987.


\bibitem[Ka90]{katz90}
Nicholas M. Katz, {\it Exponential sums and differential equations}. Annals of Math. Studies. Princeton University Press. 1990. 

\bibitem[Ka70]{katz70} Nicholas M. Katz, {Nilpotent connections and the monodromy theorem: Applications of a result of Turrittin},  Publ. Math. IH\'ES,  No. 39, (1970), 175 -- 232. 


\bibitem[LF13]{lopez_franco13}I. L\'opez Franco, {\it Tensor products of finitely cocomplete and abelian categories}, Jour. Algebra 396 (2013) 207--219. 

\bibitem[Mac98]{maclane98}S. Mac Lane, Categories for the working mathematician, GTM 5, Second edition, Springer-Verlag, New York, 1998. 

\bibitem[Ma65]{manin65}Y. Manin, Moduli fuchsiani. (Italian) Ann. Scuola Norm. Sup. Pisa Cl. Sci. (3) 19 (1965), 113--126. 
\bibitem[Mat89]{matsumura89} H. Matsumura, \emph{Commutative ring theory}, Cambridge Studies in Advanced Mathematics, Cambridge University Press, 1989.

\bibitem[Mat80]{matsumura80} H. Matsumura. Commutative Algebra. Second edition. Mathematics Lecture Note Series, 56. Benjamin/Cummings Publishing Co., Inc., Reading, Mass., 1980 


\bibitem[MFS19]{mfs}T. Monteiro Fernandes and C. Sabbah, {\it Riemann-Hilbert correspondence for mixed twistor $\cd$-modules} . J. Inst. Math. Jussieu 18 (2019), no. 3, 629--672. 

\bibitem[MFS13]{mfs0}T. Monteiro Fernandes and C. Sabbah, {\it
On the de Rham complex of mixed twistor D-modules}, Int.
Int. Math. Res. Not. IMRN Vol. 2013 (2013) 4961--4984. 


\bibitem[N93]{nitsure1} N. Nitsure, {\it Moduli of semistable logarithmic connections}, J. Amer. Math. Soc. 6 (1993),
597--609.




\bibitem[vdPS03]{van_der_put-singer03}
M. van der Put and M. Singer, {\it Galois theory of linear differential equations}. Grundlehren der mathematischen Wissenschaften 328. Springer 2003. 



\bibitem[Sa72]{saavedra72} N. Saavedra Rivano, {\it Cat\'egories Tannakiennes}. Lecture Notes in Mathematics 265, Springer-Verlag, Berlin (1972).

\bibitem[dS21]{dos_santos21tp}J. P. dos Santos, {\it A complement to Deligne's tensor product of categories}. To appear. 

\bibitem[dS09]{dos_santos09}J. P. dos Santos, {\it The behaviour of the differential Galois group on the generic and special fibres:
A Tannakian approach }.  J. Reine Angew. Math. 637 (2009), 63--98.   


\bibitem[Sh00]{shiho00}A. Shiho, {\it Crystalline fundamental groups I --- Isocrystals on the log crystalline site and log convergent site}. J. Math. Sci. Univ. Tokyo {\bf7} (2000), 509--656.

 
\bibitem[Wa76]{wasow76}W. Wasow, Asymptotic expansions for ordinary differential equations. Robert E. Krieger Publishing Co. Huntington, New York. Reprint of the first edition with corrections. 1976. 


\bibitem[Wat79]{waterhouse79} William C. Waterhouse, {\it Introduction to affine group schemes}, Graduate Texts in Mathematics, 66. Springer-Verlag, New York-Berlin, 1979. 


\bibitem[EGA]{ega} A. Grothendieck, in collaboration with J. Dieudonn\'e. {\it \'El\'ements de G\'eom\'etrie Alg\'ebrique}. Publ. Math.   IHES 8, 11 (1961), 17 (1963), 20 (1964), 24 (1965), 28 (1966), 32 (1967). 


\end{thebibliography}
\end{document}